\documentclass[12pt,reqno]{amsart}
\usepackage[a4paper,top=3cm,bottom=3cm,inner=3cm,outer=3cm]{geometry}

\usepackage{amsmath, amssymb, amsfonts, mathtools}
\usepackage{wasysym}

\usepackage{xcolor}
\definecolor{darkgreen}{rgb}{0,0.45,0}
\usepackage[
    colorlinks, 
    citecolor=blue, 
    urlcolor=darkgreen, 
    final, 
    hyperindex, 
    pagebackref, 
    linkcolor = blue
]{hyperref}

\usepackage[capitalize]{cleveref}
\usepackage{xspace}
\usepackage{youngtab}
\usepackage{lmodern}
\usepackage{amsthm}

\theoremstyle{plain}
\newtheorem{thm}{Theorem}
\newtheorem*{thm*}{Theorem}
\newtheorem{prop}[thm]{Proposition}
\newtheorem{lem}[thm]{Lemma}
\newtheorem*{prop*}{Proposition}
\newtheorem{cor}[thm]{Corollary}
\newtheorem{conj}[thm]{Conjecture}

\theoremstyle{definition}
\newtheorem{defn}[thm]{Definition}
\newtheorem{example}[thm]{Example}

\theoremstyle{remark}

\theoremstyle{plain} 

\newcommand{\thistheoremname}{}
\newtheorem*{genericthm*}{\thistheoremname}
\newenvironment{namedthm*}[1]
  {\renewcommand{\thistheoremname}{#1}%
   \begin{genericthm*}}
  {\end{genericthm*}}
\title[2-Rig Extensions and The Splitting Principle]{2-Rig Extensions and the Splitting Principle}

\date{\today}
\author[Baez]{John C.\ Baez$^1$} 
\author[Moeller]{Joe Moeller$^2$}
\author[Trimble]{\\Todd Trimble$^3$}

\address{$^1$Department of Mathematics, University of California, Riverside, CA 92521, USA}
\address{$^2$California Institute of Technology, Pasadena, CA 91125, USA}
\address{$^3$Department of Mathematics, Western Connecticut State University, Danbury, CT 06810, USA}

\email{baez@math.ucr.edu, jmoeller@caltech.edu, trimblet@wcsu.edu}

\DeclareFontFamily{U}{min}{}
\DeclareFontShape{U}{min}{m}{n}{<-> udmj30}{}

\newcommand{\define}[1]{{\bf \boldmath{#1}}\index{#1}}

\newcommand{\maps}{\colon}
\newcommand{\To}{\Rightarrow}

\newcommand{\op}{^\mathrm{op}}

\newcommand{\ev}{\mathrm{ev}}
\newcommand{\coev}{\mathrm{coev}}

\renewcommand{\hom}{\mathrm{hom}}

\newcommand{\category}[1]{\mathsf{#1}}
\newcommand{\A}{\category A}
\newcommand{\B}{\category B}
\newcommand{\C}{\category C}
\newcommand{\D}{\category D}
\newcommand{\E}{\category E}

\newcommand{\K}{\category K}

\newcommand{\M}{\mathrm{M}}
\newcommand{\N}{\mathbb N}

\newcommand{\R}{\category R}
\renewcommand{\S}{\category S}

\newcommand{\Z}{\mathbb Z}

\newcommand{\GL}{\mathrm{GL}}
\newcommand{\BU}{\mathrm{BU}}
\newcommand{\U}{\mathrm{U}}
\newcommand{\BT}{\mathrm{BT}}
\newcommand{\T}{\mathrm{T}}
\newcommand{\CC}{\mathbb{C}}
\newcommand{\Q}{\mathbb{Q}}

\newcommand{\namedcat}[1]{\mathsf{#1}}

\newcommand{\Ainfty}{\mathsf{A}^{\boxtimes \infty}}
\newcommand{\Ab}{\namedcat{Ab}}
\newcommand{\Alg}{\namedcat{Alg}}\newcommand{\Aff}{\namedcat{AffSch}}

\newcommand{\Comm}{\namedcat{Comm}}

\newcommand{\End}{\namedcat{End}}

\newcommand{\Fin}{\namedcat{Fin}}

\newcommand{\CMon}{\namedcat{CMon}}

\newcommand{\Rep}{\namedcat{Rep}}

\newcommand{\Set}{\namedcat{Set}}
\newcommand{\SMC}{\namedbicat{SMCat}}
\newcommand{\SMLin}{\namedbicat{SMLin}}

\newcommand{\TRig}{2\mhyphen\namedbicat{Rig}}

\newcommand{\Vect}{\namedcat{Vect}}
\newcommand{\VB}{\namedcat{VB}}
\newcommand{\Coh}{\namedcat{Coh}}

\newcommand{\namedbicat}[1]{\mathbf{#1}}

\newcommand{\AAff}{\namedbicat{AffSch}}

\newcommand{\CCat}{\namedbicat{Cat}}

\newcommand{\FFin}{\namedbicat{Fin}}

\newcommand{\Lin}{\namedbicat{Lin}}
\newcommand{\Cauch}{\namedbicat{Cauch}}

\newcommand{\john}[1]{\textcolor{purple}{#1}}
\newcommand{\todd}[1]{\textcolor{darkgreen}{#1}}

\mathchardef\mhyphen="2D

\newcommand{\G}{\mathsf{G}} 

\newcommand{\ksbar}{\overline{k\S}}
\newcommand{\Sym}{\mathrm{S}}
\newcommand{\triv}{\mathrm{triv}}
\newcommand{\sign}{\mathrm{sgn}}
\newcommand{\Spec}{\mathrm{Spec}}

\DeclareMathOperator{\im}{im}

\usepackage{tikz, tikz-cd}
\usetikzlibrary{positioning}
\definecolor {processblue}{cmyk}{0.9,0.5,0,0}

\tikzstyle{simple}=[-,line width=2.000]
\tikzstyle{arrow}=[-,postaction={decorate},decoration={markings,mark=at position .5 with {\arrow{>}}},line width=1.100]
\pgfdeclarelayer{edgelayer}
\pgfdeclarelayer{nodelayer}
\pgfdeclarelayer{bg}
\pgfsetlayers{bg,edgelayer,nodelayer,main}
\tikzstyle{none}=[inner sep=-1pt]
\tikzstyle{species}=[circle,fill=none,draw=black,scale=1.0]
\tikzstyle{transition}=[rectangle,fill=none,draw=black,scale=1.15]
\tikzstyle{empty}=[circle,fill=none, draw=none]
\tikzstyle{inputdot}=[circle,fill=black,draw=black, scale=.5]
\tikzstyle{dot}=[circle,fill=black,draw=black]
\tikzstyle{bounding}=[circle,dashed, fill=none,draw=black, scale=9.00]
\tikzstyle{simple}=[-,draw=black,line width=1.000]
\tikzstyle{inarrow}=[-,draw=black,postaction={decorate},decoration={markings,mark=at position .5 with {\arrow{>}}},line width=1.000]
\tikzstyle{tick}=[-,draw=black,postaction={decorate},decoration={markings,mark=at position .5 with {\draw (0,-0.1) -- (0,0.1);}},line width=1.000]
\tikzstyle{inputarrow}=[->,draw=black, shorten >=.05cm]

\tikzset{main node/.style={circle,fill=blue!20,draw,minimum size=1cm,inner sep=0pt},}

\tikzstyle{construct}=[fill=white, draw=black, shape=circle]
\tikzstyle{universal}=[fill=black, draw=black, shape=circle]

\tikzset{
    labl/.style={anchor=north, rotate=90, inner sep=.5mm}
}

\tikzset{
    labl2/.style={anchor=south, rotate=90, inner sep=.5mm}
}

\numberwithin{thm}{section}

\begin{document}
\begin{abstract}
Classically, the splitting principle says how to pull back a vector bundle in such a way that it splits into line bundles and the pullback map induces an injection on $K$-theory.  Here we categorify the splitting principle and generalize it to the context of 2-rigs.  A 2-rig is a kind of categorified `ring without negatives', such as a category of vector bundles with $\oplus$ as addition and $\otimes$ as multiplication.  Technically, we define a 2-rig to be a Cauchy complete $k$-linear symmetric monoidal category where $k$ has characteristic zero.  We conjecture that for any suitably finite-dimensional object $r$ of a 2-rig $\R$, there is a 2-rig map $E \maps \R \to \R'$ such that $E(r)$ splits as a direct sum of finitely many `subline objects' and $E$ has various good properties: it is faithful, conservative, essentially injective, and the induced map of Grothendieck rings $K(E) \maps K(\R) \to K(\R')$ is injective.  We prove this conjecture for the free 2-rig on one object, namely the category of Schur functors, whose Grothendieck ring is the free $\lambda$-ring on one generator, also known as the ring of symmetric functions.  We use this task as an excuse to develop the representation theory of affine categories---that is, categories enriched in affine schemes---using the theory of 2-rigs.
\end{abstract}

\maketitle
\setcounter{tocdepth}{1} 
\tableofcontents

\newpage 

\section{Introduction}
\label{sec:intro}

The splitting principle is a fundamental concept in algebraic topology, representation theory, and algebraic geometry. It allows us to study a complicated object by finding a larger category in which it splits as a direct sum of simpler ones. Classically, the splitting principle has been used to study the Grothendieck ring of vector bundles on a connected topological space $X$. This is done by pulling back a vector bundle $E$ over $X$ along some map $\phi \maps Y \to X$ such that $\phi^\ast E$ is isomorphic to a sum of line bundles.  Our aim here is to place the splitting principle in a broader context, namely the context of 2-rigs.

A `2-rig' (over a field $k$) is a symmetric monoidal $k$-linear category that is Cauchy complete, and we consider the case where $k$ is a field of characteristic zero. Examples of 2-rigs include categories of vector bundles, group representations, and coherent sheaves.  In the classical vector bundle case, the splitting principle can be expressed as finding for any object $E$ in the 2-rig $\Vect(X)$ of vector bundles over $X$ a new 2-rig $\Vect(Y)$ of vector bundles over some space $Y$, together with a 2-rig map $\phi^\ast \maps \Vect(X) \to \Vect(Y)$ that sends $E$ to a direct sum of line objects.  A splitting principle for 2-rigs would then say that any 2-rig containing a suitably finite object can be extended to a 2-rig in which the object can be split into a sum of `subline objects'---a generalization of line bundles which we define.

\begin{namedthm*}{Splitting Principle for 2-Rigs (Conjecture)}
    Let $\R$ be a 2-rig and $r \in \R$ an object of finite subdimension. Then there exists a 2-rig $\R'$ and a map of 2-rigs $E \maps \R \to \R'$ such that:
    \begin{enumerate}
    \item $E(r)$ splits as a direct sum of finitely many subline objects.
    \item $E \maps \R \to \R'$ is a \define{2-rig extension}: it is faithful, conservative (i.e.\ it reflects isomorphisms), and essentially injective.
    \item $K(E) \maps K(\R) \to K(\R')$ is  injective.  
    \end{enumerate}
\end{namedthm*}


 Items (1) and (2) are not only a generalization of the usual splitting principle from vector bundles to arbitrary 2-rigs: they are also a \emph{categorification}, since they apply at the level of 2-rigs rather than their Grothendieck rings. The 2-rig map $E \maps \R \to \R'$ induces ring homomorphisms between their Grothendieck rings, $K(E) \maps K(\R) \to K(\R')$, and item (3) claims this is injective.  This injectivity is used to prove equations in $K(\R)$ using calculations in $K(\R')$.   Item (2) states a categorified version of this injectivity, which only implies item (3) under restricted conditions.

In this paper we prove the above conjecture for the universal example, namely the free 2-rig on one generating object. This is equivalent to the category of Schur functors \cite{Schur}, but we call it $\ksbar$ since it can be obtained by the following three-step process:
\begin{itemize}
    \item First form the free symmetric monoidal category on one generating object $x$. This is equivalent to groupoid of finite sets and bijections, which we call $\S$, with disjoint union providing the symmetric monoidal structure.
    \item Then form the free $k$-linear symmetric monoidal category on $\S$ by freely forming $k$-linear combinations of morphisms. This is called $k\S$.
    \item Then Cauchy complete $k\S$. The result, $\ksbar$, is the coproduct, as Cauchy complete $k$-linear categories, of the categories of finite-dimensional representations of all the symmetric groups $S_n$.
\end{itemize}
We describe a 2-rig map 
\[  F \maps \ksbar \to \Ainfty \]
from $\ksbar$ to the limit 
\[   \Ainfty := \lim_{\longleftarrow} \A^{\boxtimes N} \]
where $\A^{\boxtimes N}$ is our name for the free 2-rig on $N$ subline objects, say $s_1, \dots, s_N$.  The 2-rig $\Ainfty$ 
contains infinitely many subline objects $s_1, s_2, s_3, \dots $,
and the 2-rig map $F$ is characterized by the fact that it sends the generating object $x \in \ksbar$ to the infinite coproduct $s_1 \oplus s_2 \oplus \cdots$.   One of our main results, \cref{thm:splitting2}, analyzes the properties of this 2-rig map: 

\begin{thm*}
The 2-rig map $F \maps \ksbar \to \Ainfty$ is an extension of 2-rigs.
\end{thm*}

\noindent
This categorifies a classical result, namely that the free $\lambda$-ring on one generator can be identified with the $\lambda$-ring $\Lambda$ of \define{symmetric functions}: elements of $\Z[[x_1, x_2, \dots]]$ that are of bounded degree and invariant under all permutations of the variables. 

It is worth recalling how this classical result is connected to the splitting principle for vector bundles.  The classifying space $\BU$ of the infinite-dimensional unitary group 
\[   \U = \lim_{\longrightarrow} \, \U(n) \]
has the property that $K(\BU)$ is the free $\lambda$-ring on one generator.  But if one defines a subgroup $\T \subset \U$ by 
\[   \T = \lim_{\longrightarrow} \, \T^n \]
where $\T^n$ is the maximal torus of $\U(n)$, then one can show $K(\T)$ is the subring of $\Z[[x_1, x_2, \dots ]]$ consisting of power series of bounded degree.  Furthermore, the inclusion of $\T$ in $\U$ gives a map $\phi \maps \BT \to \BU$ for which the map of $\lambda$-rings $K(\phi) \maps K(\BU) \to K(\BT)$ is injective and its image is $\Lambda$.   In this situation we can also take the universal $n$-dimensional vector bundle over $\BU$, pull it back along $\phi$, and split the resulting bundle into a sum of line bundles whose $K$-theory classes are $x_1, \dots, x_n$.

Hazewinkel \cite{Hazewinkel} has written of $\Lambda$ that ``It seems unlikely that there is any object in mathematics richer and/or more beautiful than this one.'' But the $\lambda$-ring structure on symmetric functions is often described using rather complicated and unintuitive formulas in terms of symmetric functions.  Our theorem above lets us prove that symmetric functions form the free $\lambda$-ring on one generator in a more conceptual way. The Grothendieck group $K(\R)$ of any 2-rig $\R$ is a $\lambda$-ring, and any map of 2-rigs induces a map of $\lambda$-rings. In \cite{Schur} we used the fact that $\ksbar$ is the free 2-rig on one generator to prove that $\K(\ksbar)$ is the free $\lambda$-ring on one generator.   The 2-rig map $F \maps \ksbar \to \Ainfty$ induces an inclusion of $\lambda$-rings
\[   K(F) \maps K(\ksbar) \to K(\Ainfty) .\]
In \cref{thm:K(Ainfty)} we prove that $K(\Ainfty)$ is the subring of $\Z[[x_1, x_2, \dots]]$ consisting of power series of bounded degree. In \cref{thm:injection_on_K-theory} we show that $K(F)$ is an injection and its range consists of symmetric functions.  Thus, symmetric functions form the free $\lambda$-ring on one generator.

However, proving this classical result is not our main goal.  More important is to \emph{categorify} this result, and arguably still more important is to set the categorified result into a broader theory of 2-rig extensions.  We prove that the map $F$ fits into this diagram of 2-rig maps, which commutes up to natural isomorphism:
\[
\begin{tikzcd}
    & & \Ainfty \arrow[d, "\pi_N"]
    \\
    \ksbar
    \arrow[r,"A"] \arrow[rru, bend left = 20, "F"]
    & 
    \Rep(\M(N,k)) 
    \ar[r,"B"]
    \ar[d,"D"']
    &
    \Rep(k^N) \simeq \A^{\boxtimes N}
    \arrow[d,"C"]
    \\
    & 
    \Rep(\GL(N,k))
    \ar[r,"E"']
    & 
    \Rep({k^\ast}^N) 
\end{tikzcd}
\]
Here $k^N$ is the multiplicative monoid of diagonal $N \times N$ matrices with entries in $k$, while ${k^\ast}^N$ is the multiplicative group of invertible diagonal $N \times N$ matrices.  The representation categories in this diagram are 2-rigs of algebraic representations of `affine monoids': monoids in the category of affine schemes.  Thus, a substantial part of our work consists of developing the representation theory of affine monoids, and more general affine categories, from the viewpoint of 2-rigs.  

We prove that all the maps in the above square are 2-rig extensions. We then prove that $F$ is an extension using all the other maps in this diagram. For a more thorough overview of this aspect of the paper, see \cref{sec:network}.

\subsection*{Notation}
We use sans-serif font for 1-categories such as $\Vect$, and bold serif font for 2-categories such as $\TRig$. 

\section{Line and subline objects}
\label{sec:line}

To proceed we need a general theory of subline objects in any 2-rig. However, it is helpful to start with line objects. These can be defined in any symmetric monoidal category:

\begin{defn} An object $\ell$ in a symmetric monoidal category is called a \define{line object} if there exists an object $\ell^\ast$ such that $\ell \otimes \ell^\ast \cong I$, where $I$ is the unit object for the tensor product.
\end{defn}

Such objects are also called `invertible' \cite{TensorCategories}. For example, the line objects in $\Vect$ are the 1-dimensional vector spaces, the line objects in $\VB(X)$ for a topological space $X$ are the line bundles, the line objects in $\Coh(X)$ for a variety $X$ are the invertible sheaves, and the line objects in $\Rep(G)$ for a group $G$ are the 1-dimensional representations.

It is well known \cite[Sec.\ 5]{BaezLauda} that if $\ell$ is a line object we can always find morphisms $\ev \maps \ell^* \otimes \ell \to I$ and $\coev \maps I \to \ell \otimes \ell^*$ such that the following diagrams commute: 
\[
\begin{tikzcd}
    &
    \ell \otimes \ell^* \otimes \ell
    \arrow[dr, "1 \otimes \ev"]
    \\
    \ell
    \arrow[ur, "\coev \otimes 1"]
    \arrow[rr, "1"']
    &&
    \ell
\end{tikzcd}
\qquad
\begin{tikzcd}
    \ell^*
    \arrow[rr, "1"]
    \arrow[dr, swap, "1 \otimes \coev"]
    &&
    \ell^*
    \\&
    \ell^* \otimes \ell \otimes \ell^*
    \arrow[ur, swap, "\ev \otimes 1"]
\end{tikzcd}
\]

Here we are mainly concerned with line objects in 2-rigs. The situation is simplest for `connected' 2-rigs. Note that for any 2-rig $\R$, the monoid of endomorphisms of the unit object $I \in \R$ is commutative by the Eckmann--Hilton argument, and a $k$-algebra because 2-rigs are linear categories. 

\begin{defn} 
    A 2-rig is \define{connected} if the  commutative $k$-algebra $\R(I, I)$ has exactly two distinct idempotents, $0$ and $1$.
\end{defn}

For example, the 2-rig $\VB(X)$ is connected if and only if the topological space $X$ is connected because $I \in \VB(X)$ is the trivial line bundle, and all the idempotents in $\VB(X)(I, I)$ are given by multiplication by continuous functions that take on only the values $0$ and $1$. For any 2-rig $\R$, splitting an idempotent $p \in \R(I,I)$ lets us write $I \cong I_1 \oplus I_2$ and then write $\R$ as a product of two 2-rigs, one with $I_1$ serving as its unit and one with $I_2$ as its unit. 

There are two kinds of line object in a connected 2-rig. To see this, recall from \cite{Schur} that any Young diagram $\lambda$ gives a functor 
\[ S_{\R, \lambda} \maps \R \to \R \]
called a Schur functor. In particular, the two 2-box Young diagrams give Schur functors called the second symmetric power $\Sym^2$ and second exterior power $\Lambda^2$, and there is a natural isomorphism $x^{\otimes 2} \cong \Sym^2(x) \oplus \Lambda^2(x)$ for any $x \in \R$. 

\begin{prop}
\label{prop:line_objects}
    Let $x$ be a line object in a connected 2-rig. Of $\Lambda^2(x)$ and $\Sym^2(x)$, one is zero and the other is $x^{\otimes 2}$.
\end{prop}

\begin{proof}
    If $\ell$ is any line object in a 2-rig, then
    \[   \R(\ell,\ell) \cong \R(I, \ell^\ast \otimes \ell) \cong \R(I,I) \]
    and one can check that the isomorphism is not just one of vector spaces, but of $k$-algebras. Suppose $x$ is a line object. Then so is $x \otimes x$, since $x^\ast \otimes x^\ast$ is an inverse for $x \otimes x$. Thus, we have
    \[   \R(x^{\otimes 2}, x^{\otimes 2}) \cong \R(I,I) \]
    as $k$-algebras. If $\R$ is connected it follows that there are exactly two distinct idempotents in $\R(x^{\otimes 2}, x^{\otimes 2})$, namely $0$ and $1$. But there are two 
    idempotents coming from projection onto the two summands in
    \[   x^{\otimes 2} \cong \Sym^2(x) \oplus \Lambda^2(x) \]
    and these idempotents sum to the identity. Thus, one of $\Sym^2(x)$ and $\Lambda^2(x)$ must be zero, and the other must be isomorphic to $x^{\otimes 2}$.
\end{proof}

In a 2-rig, an object $x$ has $\Lambda^2(x) \cong 0$ if and only if the symmetry 
\[   \sigma_{x,x} \maps x \otimes x \to x \otimes x \]
is the identity, and $\Sym^2(x) \cong 0$ if and only if $\sigma_{x,x}$ is minus the identity. However, objects $x$ that satisfy $\Lambda^2(x) \cong 0$ or $S^2(x) \cong 0$ need not be line objects: for example, the initial object $0$ obviously satisfies both properties but is not invertible; more generally, any retract of an object satisfying one or the other property also satisfies that property but might not be invertible. In many of the examples of 2-rigs that can be considered `classical', like $\Rep_k(G)$ for a finite group $G$, or vector bundles over a space $X$, it happens that short exact sequences split, whence subobjects of line objects are indeed retracts, and therefore any subobject of a line object satisfying one of these two properties again satisfies that property. This partially justifies the following terminology. 

\begin{defn} 
    An object $x$ in a 2-rig is an \define{bosonic subline object} if $\sigma_{x,x} \maps x \otimes x \to x \otimes x$ is the identity, and a \define{fermionic subline object} if $\sigma_{x,x}$ is minus the identity. 
\end{defn}

\begin{defn} 
    An object in a 2-rig is an \define{bosonic line object} if it is a bosonic subline object and also a line object. Similarly, it is a \define{fermionic line object} if it is a fermionic subline object and also a line object.
\end{defn}

\noindent
In this language, \cref{prop:line_objects} says that in a connected 2-rig, every line object is either bosonic or fermionic, but not both.

We must warn against some traps. Not every subline object is a subobject of a line object: we shall see in Examples \ref{ex:bosonic_subline_counterexample} and \ref{ex:fermionic_subline_counterexample} that this fails in the free 2-rig on a bosonic or fermionic subline object. Further, not every subobject of a bosonic or fermionic subline object is another such subline object. For example, in the 2-rig of modules of the commutative algebra $k[x,y]$, and regarding $k[x,y]$ as a bosonic line object in the obvious way, the ideal $(x, y)$ regarded as a submodule of $k[x, y]$ is not a bosonic subline object: for the symmetry map $\sigma$ on its tensor square, $\sigma(x \otimes y) \neq x \otimes y$. 
 
Bosonic line and subline objects are common in mathematics. For example, in the category $\VB(X)$ of vector bundles on a topological space $X$ a bosonic line object is the same as a line bundle, while a bosonic subline object is the same as a subobject of a line bundle: that is, a vector bundle each whose fibers has dimension $0$ or $1$. In the category $\Rep(G)$ of finite-dimensional representations of a group, a bosonic line object is a 1-dimensional representation of $G$. For example, $\Rep(S_n)$ has two bosonic line objects whenever $n > 1$: the trivial representation and the sign representation.

Fermionic line objects and fermionic subline objects are a bit more esoteric: there are never any in $\VB(X)$ or $\Rep(G)$. They appear in `supermathematics', where we replace vector spaces by super vector spaces. A super vector space is simply a $\Z/2$-graded vector space, i.e., a vector space $V = V_0 \oplus V_1$ that is split into a `bosonic part' $V_0$ and a `fermionic part' $V_1$. We can define a category of super vector bundles $\S\VB(X)$ over a topological space $X$, where a super vector bundle is a vector bundle $E \to X$ equipped with a splitting $E \cong E_0 \oplus E_1$ into a bosonic and fermionic part, and a map is a vector bundle map preserving this splitting. This category $\S\VB(X)$ is a 2-rig in which the symmetry introduces a sign change when permuting two homogeneous elements that are both of odd degree. A super vector bundle $E$ is a line object if and only if all the fibers of $E$ are 1-dimensional. If all the fibers of $E_0$ are 1-dimensional, $E$ is a bosonic line object, but if all the fibers of $E_1$ are 1-dimensional, $E$ is a fermionic line object. If $X$ is not connected, there are also line objects in $\S\VB(X)$ that are neither bosonic nor fermionic. 

\section{The free 2-rig on a bosonic subline object}
\label{sec:free_on_bosonic_subline}

We now describe the free 2-rig on a bosonic subline object, which we call $\A$. We can guess what this should be. It should contain a bosonic subline object $s$ and its tensor powers $s^{\otimes n}$ for all integers $n \ge 0$. We expect that 
\[   \R(s^{\otimes m}, s^{\otimes n}) \cong 
\left\{ \begin{array}{ccl}
k & \text{if} & m = n \\
0 & \text{if} & m \ne n.
\end{array}
\right.
\]
with composition being multiplication in $k$. The monoidal structure should have
\[   s^{\otimes m} \otimes s^{\otimes n} \cong s^{\otimes (m+n)} \]
and the symmetry should behave in a trivial way, since
\[    \sigma_{s,s} \maps s \otimes s \to s \otimes s \]
must be the identity, given that the projection from $s^{\otimes 2}$ to $\Sym^2(s)$ is the identity.

Not all the objects in $\A$ will be of the form $s^{\otimes n}$, since a 2-rig must have finite direct sums, and all idempotents must split. However, once we adjoin finite direct sums of the objects $s^{\otimes n}$, all idempotents will split.

This leads us to the following simple definition of $\A$. As a category, it is the category of $\N$-graded vector spaces with finite total dimension, and linear maps preserving the grading. We give this category its usual linear structure. We also give it the usual monoidal structure, where
\[    (V \otimes W)_n = \bigoplus_{i+j = n} V_i \otimes W_j .\]
We let $s \in \A$ be the graded vector space with
\[   s_n = 
\left\{ \begin{array}{ccl}
k & \text{if} & n = 1 \\
0 & \text{if} & n \ne 1.
\end{array}
\right.
\]
There are two possible symmetries compatible with the monoidal structure described. For our purposes we need $s$ to be a bosonic subline object, so we need
\[    \sigma_{s, s} = 1_{s \otimes s} .\]
Thus, we use the symmetry on $\A$ where
\[    \sigma_{V,W} \maps V \otimes W \to W \otimes V \]
is defined using the usual symmetry in $\Vect$ on each homogeneous component:
\[    (V \otimes W)_n = \bigoplus_{i+j = n} V_i \otimes W_j \longrightarrow \bigoplus_{i+j = n} W_j \otimes V_i = (W \otimes V)_n .\]
In \cref{thm:free_2-rig_on_fermionic_subline} we discuss another choice of symmetry, which gives the free 2-rig on a fermionic subline object.

\begin{thm} 
\label{thm:free_2-rig_on_bosonic_subline}
    $\A$ is the free 2-rig on a bosonic subline object. That is, given a 2-rig $\R$ containing a bosonic subline object $x$, there is a map of 2-rigs $F \maps \A \to \R$ with $F(s) = x$, and $F$ is determined uniquely up to isomorphism by this property.
\end{thm}

Before giving the proof, we recall some notation. The category $\Fin\Vect$ of finite-dimensional vector spaces is the initial 2-rig, so for any 2-rig $\R$, we have a unique 2-rig map $i_\R \maps \Fin\Vect \to \R$. Given a vector space $V$ and an object $R$ of $\R$, we let $V \cdot R$ denote the tensor product $i_\R(V) \otimes R$. 

\begin{proof}
An object $V$ of $\A$ is given by its homogeneous components, $V= (V_n)_{n \geq 0}$. We define the 2-rig map $F$ by
\[
F(V) = \bigoplus_{n \geq 0} V_n \cdot x^{\otimes n}
\] 
In particular, $F(s) = x$, and it is straightforward to check that we have canonical isomorphisms 
\[
\begin{array}{ccl}
F(V \otimes W) & \cong & \displaystyle{ \sum_{n \geq 0} \left(\bigoplus_{j+k=n} V_j \otimes W_k\right) \cdot x^{\otimes n} } \\ \\
& \cong &  \displaystyle{  \left(\bigoplus_{j \geq 0} V_j \cdot x^{\otimes j}\right) \otimes \left(\bigoplus_{k \geq 0} W_k \cdot x^{\otimes k}\right) } \\ \\
& = & F(V) \otimes F(W)
\end{array}
\] 
making $F$ into a strong monoidal functor. One can check that $F$ is symmetric monoidal, and that $F$ is unique up a monoidal natural isomorphism. In other words, this definition of $F$ as a 2-rig map is forced on us: writing an arbitrary object of $\A$ as 
\[\bigoplus_{n \geq 0} V_n \cdot s^{\otimes n} \cong \bigoplus_{n \geq 0} i_\A(V_n) \otimes s^{\otimes n},
\] 
we have  
\begin{align*}
    F\left(\bigoplus_{n \geq0} i_\A(V_n) \otimes s^{\otimes n}\right) 
    &\cong \bigoplus_{n \geq 0} F(i_\A(V_n) \otimes s^{\otimes n}) 
    & F \text{ preserves coproducts}
    \\& \cong \bigoplus_{n \geq 0} F(i_\A(V_n)) \otimes F(s)^{\otimes n}
    & F \text{ preserves tensor products}
    \\& \cong \bigoplus_{n \geq 0} i_\R(V_n) \otimes F(s)^{\otimes n} 
    & \text{$\Fin\Vect$ is the initial 2-rig} 
    \\& \cong \bigoplus_{n \geq 0} V_n \cdot x^{\otimes n} & F(s) = x.
\end{align*}
\end{proof}

While this result is straightforward, it is interesting to set it in a larger context. The concept of `bosonic subline object' makes sense in any symmetric monoidal category: it is an object $x$ with $\sigma_{x,x} = 1_{x \otimes x}$. One can show that the free symmetric monoidal category on a bosonic subline object is the discrete category on $\N$ with addition as its monoid operation. We call this symmetric monoidal category simply $\N$. Then, to obtain $\A$, we can apply two left 2-adjoints introduced in \cite[Thm.\ 3.5]{Schur}:
\[
\begin{tikzcd}
    &
    \SMC
    \arrow[r, bend left, "k(-)"]
    \arrow[r, phantom, "\bot"]
    &
    \SMLin
    \arrow[l, bend left, "U_1"]
    \arrow[r, bend left, "\overline{(-)}"]
    \arrow[r, phantom, "\bot", pos = 0.4]
    &
    \TRig
    \arrow[l, bend left, "U_2", pos = 0.45]
\end{tikzcd}
\]
The functor $k(-) \maps \SMC \to \SMLin$ performs base change, freely turning any symmetric monoidal category $\C$ into a symmetric monoidal $k$-linear category $k(\C)$, which has the same objects as $\C$ but with hom-spaces being the the free vector spaces on the homsets of $\C$. The functor $\overline{(-)} \maps \SMLin \to \TRig$ performs Cauchy completion, freely endowing any symmetric monoidal $k$-linear category with absolute colimits. Given that $\N \in \SMC$ is the free symmetric monoidal category on a bosonic line object, it follows that $\overline{k\Z}$ is the free 2-rig on a bosonic line object. One can then  check that $\overline{k\Z}$ is equivalent, as a 2-rig, to $\A$.

Finally, we note a counterexample:

\begin{example}
\label{ex:bosonic_subline_counterexample}
Not every bosonic subline object is a subobject of a line object. For since the tensor product in $\A$ is the usual tensor product of $\N$-graded vector spaces, the only line object in $\A$ is the tensor unit $I$. Thus, the bosonic subline object $s \in \A$ is not a subobject of any line object.
\end{example}

\section{The free 2-rig on a bosonic line object}
\label{sec:free_on_bosonic_line}

Having described the free 2-rig on a bosonic subline object, it is easy to adapt our treatment to describe the free 2-rig on a bosonic line object, which we call $\T$. This should contain a bosonic line object $\ell$ and all its tensor powers $\ell^{\otimes n}$, but because a line object has an inverse object, $n$ can now be negative as well as positive.

Thus, we define $\T$ to be the category of $\Z$-graded vector spaces of finite total dimension. We give this category its usual linear structure and its usual monoidal structure, where
\[    (V \otimes W)_n = \bigoplus_{i+j = n} V_i \otimes W_j .\]
We let $\ell \in \T$ be the graded vector space with
\[   \ell_n = 
\left\{ \begin{array}{ccl}
k & \text{if} & n = 1 \\
0 & \text{if} & n \ne 1.
\end{array}
\right.
\]
As before, there are two possible symmetries compatible with the monoidal structure. To ensure that $\ell$ is a bosonic line object, we use the symmetry
\[    \sigma_{V,W} \maps V \otimes W \to W \otimes V \]
defined using the usual symmetry in $\Vect$ on each homogeneous component:
\[    (V \otimes W)_n = \bigoplus_{i+j = n} V_i \otimes W_j \longrightarrow \bigoplus_{i+j = n} W_j \otimes V_i = (W \otimes V)_n .\]

\begin{thm} 
\label{thm:free_2-rig_on_bosonic_line}
    $\T$ is the free 2-rig on a bosonic line object. That is, given a 2-rig $\R$ containing a bosonic line object $x$, there is a map of 2-rigs $F \maps \T \to \R$ with $F(\ell) = x$, and $F$ is determined uniquely up to natural isomorphism by this property.
\end{thm}

\begin{proof}
An object $V$ of $\T$ is given by its homogeneous components, $V= (V_n)_{n \in \Z}$. We define the 2-rig map $F$ by
\[
F(V) = \bigoplus_{n \in \Z} V_n \cdot x^{\otimes n}
\] 
where a negative tensor power of $x$ is defined to be the corresponding positive tensor power of the inverse object $x^*$.
In particular, $F(\ell) = x$, and as in the proof of \cref{thm:free_2-rig_on_bosonic_subline} there are canonical isomorphisms 
\[   F(V \otimes W) \stackrel{\sim}{\longrightarrow} F(V) \otimes F(W)
\] 
making $F$ into a strong monoidal functor. One can check that $F$ is symmetric monoidal, and also that $F$ is unique up a monoidal natural isomorphism:
\begin{align*}
    F\left(\bigoplus_{n \in \Z} i_\T(V_n) \otimes \ell^{\otimes n}\right) 
    &\cong \bigoplus_{n \in \Z} F(i_\T(V_n) \otimes \ell^{\otimes n}) 
    & F \text{ preserves coproducts}
    \\& \cong \bigoplus_{n \in \Z} F(i_\T(V_n)) \otimes F(\ell)^{\otimes n}
    & F \text{ preserves tensor products}
    \\& \cong \bigoplus_{n \in \Z} i_\R(V_n) \otimes F(\ell)^{\otimes n} 
    & \text{$\Fin\Vect$ is the initial 2-rig} 
    \\& \cong \bigoplus_{n \in \Z} V_n \cdot x^{\otimes n} & F(\ell) = x.
\end{align*}
\end{proof}

\section{Algebraic representations}

We shall be studying 2-rigs of representations of various monoids, such as the monoid of $n \times n$ matrices, the group of invertible $n \times n$ matrices, and so on. However, we will typically not consider general representations, only so-called `algebraic' ones. This restriction lets us avoid the vast wilderness of representations that arise from automorphisms of the ground field $k$. For example, the multiplicative monoid of $k$ has a 1-dimensional representation coming from any automorphism of the field $k$, but only for the identity automorphsm is this representation algebraic.

To define this concept of `algebraic' representation, we use the fact that the monoids we are considering are actually monoid objects in the category of affine schemes. Following Milne \cite[Chap.\ I.4]{MilneAffGpSch} we call such things `affine monoids'.

\begin{defn} The category of \define{affine schemes} over $k$, $\Aff$, is the opposite of the category $\Comm\Alg$ of commutative algebras over $k$. Since $\Aff$ has cartesian products we can define monoids internal to it, and an \define{affine monoid} is a monoid internal to $\Aff$. The
\end{defn}

\begin{lem}
\label{lem:commutative_bialgebras_as_affine_monoids}
The category of affine monoids is equivalent to the category of commutative bialgebras over $k$.
\end{lem}

\begin{proof}
An affine monoid is a monoid in $(\Aff, \times) \simeq (\Comm\Alg, \otimes)\op$, where $\otimes$ denotes the tensor product of commutative algebras over $k$, which is the coproduct in $\Comm\Alg$. But a monoid in $(\Comm\Alg, \otimes)\op$ is the same as a comonoid in $(\Comm\Alg, \otimes)$, which is a commutative bialgebra over $k$. 
\end{proof}

\begin{defn}
\label{defn:coordinate_algebra}
If $M$ is an affine monoid, we call the corresponding commutative bialgebra its \define{coordinate bialgebra} $\mathcal{O}(M)$.
\end{defn}

\begin{example}
\label{ex:commutative_monoids_as_affine_monoids}
Any commutative monoid $M$ gives rise to an affine monoid $\Spec(kM)$ whose coordinate bialgebra is the monoid algebra $kM$. In more detail, the free vector space functor 
\[
    \begin{array}{ccl} \Set &\to& \Vect \\
    X &\mapsto& kX \end{array}
\]
is strong symmetric monoidal, hence takes cocommutative comonoids in $(\Set, \times)$, which are simply sets, to cocommutative comonoids in $(\Vect, \otimes)$, which are cocommutative coalgebras. By the same reasoning, it takes bicommutative bimonoids in $(\Set, \times)$, which are the same as commutative monoids $M$, to bicommutative bimonoids $kM$ in $(\Vect, \otimes)$, which are bicommutative bialgebras. We have 
\[
    k(M \times N) \cong kM \otimes kN
\]
where the right side is the coproduct in the category of bicommutative bialgebras. Meanwhile, $M \times N$ is the biproduct, hence coproduct, of the commutative monoids $M, N$. Thus we have a coproduct preserving functor $M \mapsto kM$ from commutative monoids to (bi)commutative bialgebras. Taking opposites of categories, where commutative bialgebras are opposite to affine monoids, the induced functor 
\[
    \Spec(k-) \maps \mathsf{CMon}\op \to \mathsf{AffMon}
\]
preserves products. 
\end{example}

\begin{example} 
\label{ex:algebras_as_affine_monoids}
Any finite-dimensional algebra over $k$ gives rise to an affine monoid. We shall need more general facts of a related nature, so it is worth going into some detail here. First notice that there is a symmetric lax monoidal functor
\[   \Phi \maps (\Fin\Vect, \otimes) \to (\Aff, \times) \]
given as the composite
\[    (\Fin\Vect, \otimes) \xrightarrow{(-)^\ast} 
(\Fin\Vect, \otimes)\op \xrightarrow{S\op} (\Comm\Alg, \otimes)\op = (\Aff, \times) \]
In the first step, taking the dual is a symmetric strong monoidal functor from $(\Fin\Vect, \otimes)$ to $ (\Fin\Vect, \otimes)\op$. In the second step, $\Sym \maps (\Fin\Vect, \otimes) \to (\Comm\Alg, \otimes)$ sends any finite-dimensional vector space $V$ to the free commutative algebra on $V$, also known as the symmetric algebra $\Sym(V)$. Being left adjoint to the forgetful functor $U \maps (\Comm\Alg, \otimes) \to (\Fin\Vect, \otimes)$ which is strong symmetric monoidal, $\Sym$ is symmetric oplax monoidal, so $\Sym\op \maps (\Fin\Vect, \otimes)\op \to (\Comm\Alg, \otimes)\op$ is symmetric lax monoidal. 

Since $\Phi$ is symmetric lax monoidal we can use it to convert monoid objects in $(\Fin\Vect,\otimes)$, which are simply finite-dimensional algebras, into affine monoids.
\end{example}

More generally we can use the symmetric lax monoidal functor
\[   \Phi \maps (\Fin\Vect, \otimes) \to (\Aff, \times) \]
to convert categories enriched in finite-dimensional vector spaces into categories enriched in affine schemes.

\begin{defn} 
\label{defn:finite-dim_linear_category}
    Let $\FFin\Lin\CCat$ be the 2-category of categories, functors and natural transformations enriched over $\Fin\Vect$. We call these \define{finite-dimensional linear} categories, \define{linear} functors and natural transformations. 
\end{defn}

\begin{defn}
\label{defn:algebraic_category}
    Let $\AAff\CCat$ be the 2-category of categories, functors and natural transformations enriched over $\Aff$. We call these \define{affine categories}, \define{algebraic functors} and natural transformations.
\end{defn}

The following lemma is then a routine consequence of the theory of base change for enriched categories \cite{Kelly}:

\begin{lem}
\label{lem:base_change}
    Base change along $\Phi$ gives a 2-functor
    \[   (-)^\sim \, \maps
    \FFin\Lin\CCat \to \AAff\CCat \]
    sending any finite-dimensional linear category $\C$ to the affine category $\C^\sim$ with the same objects, and with hom-objects defined by 
    \[   \C^\sim(x,y) = \Phi(\C(x,y)),\]
    and composition and units defined using the functoriality of $\Phi$.
\end{lem}

Crudely put, the hom-sets of $\C^\sim$ are `the same' as those of $\C$, but instead of treating them as vector spaces we treat them as affine schemes, which gives them greater flexibility: now we allow not just linear maps between them, but maps defined by arbitrary polynomials in any linear coordinates on these spaces. In particular, a one-object linear category $\C$ is just a way of thinking about an algebra over $k$, and the one-object affine category $\C^\sim$ is a way of thinking about the corresponding affine monoid.

\begin{defn} 
\label{defn:algebraic_representation}
Given an affine category $\C$ let 
\[   \Rep(\C) = \AAff\CCat(\C, \Fin\Vect^\sim). \]
The objects of $\Rep(\A)$ are algebraic functors $F \maps \C \to \Fin\Vect^\sim$, which we call \define{algebraic representations} of $\A$, and the morphisms are natural transformations between these. 
\end{defn}

We are especially interested in representations of affine monoids, which can be seen as one-object affine categories. Four examples play a major role in what follows. Two arise from algebras over $k$:

\begin{example}  
\label{ex:MNk}
Let $\M(N,k)$ be the affine monoid arising from the algebra of $N \times N$ matrices over $k$. The coordinate bialgebra of this affine monoid is the polynomial algebra on elements $e_{ij}$ ($1 \le i, j \le N$) with comultiplication 
\[   \Delta(e_{ij}) = \sum_{k = 1}^N e_{ik} \otimes e_{kj} .\] 

This example also has a useful basis-independent description. Let $V$ be a finite-dimensional vector space. The monoid (i.e.\ $k$-algebra) $\Fin\Vect(V, V)$ can be regarded as the dual $\Fin\Vect(V \otimes V^\ast, k)$ of a comonoid structure on $V \otimes V^\ast$. The comultiplication on $V \otimes V^\ast$ is given by a composite 
\[
V \otimes V^\ast \cong V \otimes k \otimes V^\ast \xrightarrow{1 \otimes \eta \otimes 1} V \otimes V^\ast \otimes V \otimes V^\ast
\]
with $\eta$ a canonical map of the form $k \to \hom(V, V) \cong V^\ast \otimes V$, where $k \to \hom(V, V)$ takes $1 \in k$ to $1_V \maps V \to V$. The element $\eta(1)$ is $\sum_{k = 1}^N f^k \otimes e_k \in V^\ast \otimes V$ where $e_1, \ldots, e_N$ is any basis of $V$ and $f^1, \ldots, f^N$ is the dual basis (but $\eta(1)$ itself is independent of basis). Note that the `matrix element' $e_{ij}$ in the prior description of the comultiplication corresponds to $e_i \otimes f^j$; comultiplication on the commutative bialgebra $\mathcal{O}(\M(N, k))$ (now recast as $S(V \otimes V^\ast)$), as described above, takes 
\[
e_i \otimes f^j \mapsto e_i \otimes \eta(1) \otimes f^j = \sum_{k=1}^N e_i \otimes f^k \otimes e_k \otimes f^j
\]
but in basis-free form, this is just $v \otimes f \mapsto v \otimes \eta(1) \otimes f$. 

In this way, we may speak directly of the affine monoid $\hom(V, V)^\sim$ for a finite-dimensional vector space $V$.
\end{example}

\begin{example}  
\label{ex:M_1}
When $N = 1$ we call the affine monoid $\M(N,k)$ simply $k$, since it arises by applying $\Phi$ to the 1-dimensional algebra $k$. We thus define $\Rep(k)$ to be $\Rep(\M(N,k))$ for $N = 1$. This affine monoid $k$ is also $\Spec(k \N)$ as defined in  \cref{ex:commutative_monoids_as_affine_monoids}.
\end{example}

Two other examples arise from affine groups:

\begin{example}
\label{ex:GL}
There is a well-known way to treat the group $\GL(N,k)$ of invertible $N \times N$ matrices as an affine group \cite{MilneAlgGp}, and we use this to define $\Rep(\GL(N,k))$. 
\end{example}

\begin{example}
\label{ex:GL_1}
When $N = 1$ we call the affine group $\GL(N,k)$ simply $k^\ast$, and we use this to define $\Rep(k^\ast)$. This affine monoid $k^\ast$ is also $\Spec(k \Z)$ as defined in  \cref{ex:commutative_monoids_as_affine_monoids}.
\end{example} 

For any affine monoid $M$ it is useful to have concrete descriptions of its algebraic representations in terms of its coordinate bialgebra.

\begin{lem}
\label{lem:comodule}
The category $\Rep(M)$ of algebraic representations of an affine monoid $M$ is equivalent to the category of finite-dimensional comodules of its coordinate bialgebra $\mathcal{O}(M)$, and the usual tensor product of comodules makes $\Rep(M)$ into a 2-rig.
\end{lem}

\begin{proof}
An algebraic representation of $M$ on a finite-dimensional vector space $V$ is an affine monoid map $M \to \hom(V, V)^\sim$ where the affine monoid $\hom(V, V)^\sim$ was described in \cref{ex:MNk}. Any such affine monoid map is the same as a map of commutative bialgebras from $S(V \otimes V^\ast)$ to $\mathcal{O}(M)$. This gives a natural transformation of monoid-valued functors
\[
\Comm\Alg(\mathcal{O}(M), -) \to \Comm\Alg(\Sym(V \otimes V^\ast), -)
\] 
By the adjunction $\Sym \dashv U$ in the 2-category of symmetric monoidal categories and oplax symmetric monoidal functors, this corresponds to a comonoid map 
\[
V \otimes V^\ast \to U(\mathcal{O}(M))
\] 
in $\Fin\Vect$, i.e., to a comodule structure $\eta: V \to V \otimes U(\mathcal{O}(M))$ over the underlying coalgebra of $\mathcal{O}(M)$. Thus, all such categories $\Rep(M)$ may be regarded as categories of comodules over coalgebras. Compare also
\cite[Sec.\ VIII.6]{MilneAffGpSch}.

Comodule categories are evidently Cauchy complete linear categories. The tensor product on $\Rep(M)$ involves the full commutative bialgebra structure on $\mathcal{O}(M)$: if $(V, \eta)$ and $(W, \theta)$ are comodules, then their tensor product is the vector space $V \otimes W$ equipped with the comodule structure given by the evident composite 
\[
V \otimes W \xrightarrow{\eta \otimes \theta} V \otimes \mathcal{O}(M) \otimes W \otimes \mathcal{O}(M) \cong V \otimes W \otimes \mathcal{O}(M) \otimes \mathcal{O}(M) \xrightarrow{1 \otimes 1 \otimes m} V \otimes W \otimes \mathcal{O}(M)
\] 
where $m$ denotes the algebra multiplication. Commutativity of $m$ ensures that the symmetry isomorphism $V \otimes W \cong W \otimes V$ of vector spaces is indeed a comodule isomorphism, and thus $\Rep(M)$ becomes a 2-rig. 
\end{proof}

Some important examples arise from this commutative square of inclusions of affine monoids:
\[
\begin{tikzcd}
    \M(N,k)
    &
    k^N
    \ar[l]
    \\
    \GL(N,k)
    \ar[u]
    & 
    {k^\ast}^N.
    \ar[u]
    \ar[l]
\end{tikzcd}
\]
Restriction of representations gives a commutative square of 2-rig maps:
\[
\begin{tikzcd}
    \Rep(\M(N,k)) 
    \ar[r]
    \ar[d]
    &
    \Rep(k^N) 
    \arrow[d]
    \\
    \Rep(\GL(N,k))
    \ar[r]
    & 
    \Rep({k^\ast}^N). 
\end{tikzcd}
\]
This square of 2-rigs plays a key role in what follows: see \cref{sec:network} for an overview. 

To conclude this section, we prove that for any affine category $\C$, the representation category
\[  \Rep(\C) = \AAff\CCat(\C, \Fin\Vect^\sim) \] 
can be given the structure of a 2-rig. We can state this result more strongly using the 2-category
$\TRig$ studied in \cite{Schur}, in which

\begin{itemize}
    \item objects are symmetric monoidal Cauchy complete $k$-linear categories (that is, 2-rigs),
    \item morphisms are symmetric monoidal $k$-linear functors (that is, maps of 2-rigs),
    \item 2-morphisms are symmetric monoidal $k$-linear natural transformations.
\end{itemize}

\begin{thm}
\label{thm:internal_2-rig}
    $\Fin\Vect$ is an `internal 2-rig' in $\AAff\CCat$, meaning that the 2-functor 
    \[ 
    \AAff\CCat(-, \Fin\Vect^\sim) \maps \AAff\CCat\op \to \CCat
    \]
    canonically lifts through the forgetful functor $\TRig \to \CCat$. 
\end{thm}

\begin{proof}
The 2-rig structure on $\AAff\CCat(\C, \Fin\Vect^\sim)$ can be defined pointwise once we transfer the 2-rig structure of $\Fin\Vect$ to the $\AAff\CCat$-enriched category $\Fin\Vect^\sim$. This 2-rig structure consists of an `additive' part (enrichment in $\Vect$, biproducts, and splitting of idempotents), and a `multiplicative' part (the symmetric monoidal structure); we consider these structures separately. 

To add morphisms $f, g \maps V \to W$ in $\Fin\Vect^\sim$, form the composite 
\[
    V \xrightarrow{\Delta} V \times V \xrightarrow{f \times g} W \times W \xrightarrow{\nabla} W
\]
where the codiagonal $\nabla$ is the addition operator, obtained from the corresponding addition operator $W \times W \to W$ in $\Fin\Vect$ by applying change of base $\Fin\Vect \to \Aff$. Similarly, to multiply $f \maps V \to W$ by a scalar $r \in k$, form the composite 
\[
    V \xrightarrow{f} W \xrightarrow{\lambda_r} W
\]
and then apply change of base to a scalar operator $\lambda_r \maps W \to W$ in $\Fin\Vect$ to obtain the corresponding operator in $\Fin\Vect^\sim$. Writing the vector space axioms diagrammatically in $\Fin\Vect$ (using $\nabla$ and the $\lambda_r$) and applying change of base, the same axioms hold in $\Fin\Vect^\sim$, hence $\Fin\Vect^\sim$ carries $\Vect$-enrichment. Similarly, to obtain biproducts of $\Fin\Vect$ considered as an $\Aff$-category, simply apply change of base $\FFin\Lin\CCat \to \AAff\CCat$ to the biproduct structure on $\Fin\Vect$, which consists of a 1-morphism or $\Fin\Vect$-enriched functor 
\[
    \oplus \maps \Fin\Vect \times \Fin\Vect \to \Fin\Vect
\]
together with various 2-morphisms needed to capture the biproduct structure, such as product projections $p_X \maps X \oplus Y \to X$ and coproduct injections $i_X \maps X \to X \oplus Y$, subject to the required 2-cell equations such as $i_X p_X + i_Y p_Y = 1_{X \oplus Y}$. Splitting of idempotent 1-cells in $\Fin\Vect^\sim$ derives from splitting of idempotents in $\Fin\Vect$. (Remember: 1-cells $V \to W$ in the underlying category of $\Fin\Vect^\sim$ `are' $k$-linear maps; they are not general maps $\Phi(V) \to \Phi(W)$ between $V$ and $W$ considered as affine schemes.) This completes the description of the additive structure of $\Fin\Vect^\sim$. 

The multiplicative structure 
    \[
    \otimes \maps \Fin\Vect^\sim \times \Fin\Vect^\sim \to
    \Fin\Vect^\sim
\]
is obtained as a composite 
\[
    \Fin\Vect^\sim \times \Fin\Vect^\sim \to (\Fin\Vect \otimes \Fin\Vect)^\sim \xrightarrow{\otimes^\sim} \Fin\Vect^\sim.
\]
The first arrow is a component of the laxator for the base change functor 
\[
    (-)^\sim \maps \FFin\Lin\CCat \to \AAff\CCat 
\]
induced by the lax symmetric monoidal functor 
\[
    \Phi \maps (\Fin\Vect, \otimes) \to (\Aff, \times)
\]
defined in \cref{ex:algebras_as_affine_monoids}. The second arrow is $(-)^\sim$ applied to $\otimes \maps \Fin\Vect \otimes \Fin\Vect \to \Fin\Vect$ in $\FFin\Lin\CCat$. In other words, we use the lax symmetric monoidal 2-functor $\FFin\Lin\CCat \to \AAff\CCat$ to map the symmetric pseudomonoid $(\Fin\Vect, \otimes)$ in $\FFin\Lin\CCat$ to a symmetric pseudomonoid $(\Fin\Vect^\sim, \otimes)$ in $\AAff\CCat$. That this multiplicative structure distributes over the additive structure boils down to its preserving $\Fin\Vect$-enrichment, which in turn follows from applying change of base to the corresponding equational statement holding in $\FFin\Lin\CCat$. 

This completes the desired lift on objects. On morphisms, the lift maps any algebraic functor $F \maps \C \to \D$ to the 2-rig map $\Rep(F) \maps \Rep(\D) \to \Rep(\C)$ given by precomposing with $F$. On 2-morphisms, it maps any natural transformation $\alpha \maps F \To G$ between algebraic functors $F, G \maps \C \to \D$ to the natural transformation given by left whiskering with $\alpha$. One can check that this lift is indeed a 2-functor.
\end{proof} 

Henceforth we use $\Rep$ to denote the lifted 2-functor in \cref{thm:internal_2-rig}. We spell out its description for future reference.

\begin{cor}
\label{cor:Rep_as_a_2-functor}
The 2-functor $\Rep \maps \AAff\CCat\op \to \TRig$ has the following properties:
\begin{itemize}
    \item It maps any affine category $\C$ to $\Rep(\C)$ made into a 2-rig as in \cref{thm:internal_2-rig}.
    \item It maps any algebraic functor $F \maps \C \to \D$ to the 2-rig map $\Rep(F) \maps \Rep(\D) \to \Rep(\C)$ given by precomposing with $F$.
    \item It maps any natural transformation $\alpha \maps F \To G$ between algebraic functors $F, G \maps \C \to \D$ to the natural transformation given by left whiskering with $\alpha$.
\end{itemize}
\end{cor}

\section{The free 2-rig on several bosonic subline objects}
\label{sec:free_2-rig_on_N_bosonic_subline_objects}

We are now in a position to describe the free 2-rig on several bosonic subline objects in two ways: an `abstract' way using 2-rig theory, and a `concrete' way using representation theory. They are, however, just slightly different outlooks on the same idea.

For the abstract description, recall from \cite[Lem.\ 4.2]{Schur} that the 2-category of 2-rigs has coproducts, with the coproduct of 2-rigs $\R$ and $\S$ denoted $\R \boxtimes \S$ because it behaves analogously to the coproduct of commutative rings, which is their usual tensor product. 

\begin{lem}
\label{lem:free_2-rig_on_N_bosonic_sublines}
The $N$-fold tensor product $\A^{\boxtimes N}$ is the free 2-rig on $N$ bosonic subline objects $s_1, \dots, s_N$. That is, given any 2-rig $\R$ containing bosonic subline objects $x_1, \dots, x_N$, there is a 2-rig map $F \maps \A^{\boxtimes N} \to \R$ with $F(s_i) = x_i$ for $1 \le i \le N$, and $F$ is determined uniquely up to isomorphism by this property.
\end{lem}

\begin{proof} As a coproduct, $\A^{\boxtimes N}$ comes equipped with coprojections $i_k \maps \A \to \A^{\boxtimes N}$ for $k = 1, \ldots, N$. Let $s_k = i_k(s)$ where $s \in \A$ is the generating bosonic subline object. By \cref{thm:free_2-rig_on_bosonic_subline}, the bosonic subline object $x_k$ in $\R$ induces a 2-rig map $F_k \maps \A \to \R$ that sends $s$ to $x_k$, uniquely up to isomorphism. By the coproduct property, all these 2-rig maps give a 2-rig map $F \maps \A^{\boxtimes N} \to \R$  with $F(s_k) = x_k$, and $F$ is determined uniquely up to isomorphism by this property.
\end{proof}

In fact the free 2-rig on a bosonic subline object is familiar, not only as the category of $\N$-graded vector spaces of finite total dimension, but as the category of algebraic representations of the affine monoid $k$ with multiplication as its monoid operation. The reason is that an algebraic representation of $k$ on a vector space $V$ corresponds to a way of making $V$ into a comodule of its coordinate bialgebra $k[x]$ by \cref{lem:comodule}. Given such a comodule 
\[    \eta \maps V \to V \otimes k[x]  \]
we can take any vector $v \in V$ and extract its homogeneous part $v_n$ in each grade $n \in \N$ by writing
\[    \eta(v) = \sum_{n \in \N} v_n \otimes x^n .\]
In the converse direction, we can make any $\N$-graded vector space into a comodule of $k[x]$ by this formula.

In more detail, following \cref{ex:commutative_monoids_as_affine_monoids}, we have the following result. 

\begin{lem}
\label{lem:graded_over_commutative_monoid}
    For any commutative monoid $M$, the 2-rig $\Rep(\Spec(kM))$ is equivalent to the 2-rig of $M$-graded vector spaces of finite total dimension, which is also the free 2-rig $\overline{kM}$ on the discrete symmetric monoidal category with elements of $M$ as objects and multiplication in $M$ as its tensor product.
\end{lem} 

\begin{proof}
By \cref{lem:comodule} we know that $\Rep(\Spec(kM))$ is equivalent to the category of finite-dimensional comodules of its coordinate bialgebra, namely the monoid algebra $kM$ equipped with the comultiplication $\delta \maps kM \to kM \otimes kM$ specified by $\delta(m) = m \otimes m$ and counit $\varepsilon \maps kM \to k$ specified by $\varepsilon(m) = 1$ for all $m \in M$. A comodule is given by a finite-dimensional vector space $V$ and a map $\eta \maps V \to V \otimes kM$. This map takes any element $v$ to an expression of type $\sum_{m \in M} v_m \otimes m$, and the counit law for the comodule amounts to the condition that $v = \sum_m v_m$, while the coassociative law amounts to the conditions that $(v_m)_n = 0$ if $m \neq n$ and $(v_m)_m = v_m$. These are exactly what is needed to say that $V$ is the total space of an $M$-graded vector space, where the homogeneous component of a vector $v$ in grade $m \in M$ is $v_m$. The same line of thought prescribes the grade of $v_p \otimes w_q$ for homogeneous elements $v_p, w_q$ in two comodules $V$ and $W$ to be $pq$, so that 
    \[
    (V \otimes W)_m = \bigoplus_{m = pq} V_p \otimes W_q,
    \]
    and the symmetry isomorphism is the usual (unsigned) switch of tensor factors, $V_p \otimes W_q \to W_q \otimes V_p$.
    
    This 2-rig of $M$-graded vector spaces of finite total dimension is equivalent to the 2-rig obtained by starting from $M$ viewed as a discrete symmetric monoidal category with one object for each element of $M$, then forming the $k$-linear symmetric monoidal category (also denoted $kM$) by applying the free vector space functor to the hom-sets for $M$, and finally closing up under biproducts and retracts to form the $k$-linear Cauchy completion $\overline{kM}$. The final assertion is that strong symmetric monoidal functors from $M$ to the underlying symmetric monoidal category of any 2-rig $\R$ are equivalent to 2-rig maps $\overline{kM} \to \R$. This follows from Lemmas 14 and 15 of \cite{Schur}. 
\end{proof}

We can use this result to describe the 2-rig of representations of the algebraic monoid $k \cong \Spec(k\N)$ introduced in \cref{ex:M_1}.

\begin{lem}
\label{lem:free_2-rig_on_bosonic_subline}
The 2-rig $\Rep(k)$ is the free 2-rig on a bosonic subline object, $\A$.
\end{lem}

\begin{proof}
By \cref{lem:graded_over_commutative_monoid}, $\Rep(k) \simeq \Rep(\Spec(k\N))$ is the 2-rig of $\N$-graded vector spaces of finite total dimension, or more precisely $\A$, which according to \cref{thm:free_2-rig_on_bosonic_subline} is the free 2-rig on a bosonic subline object.
\end{proof}

We now describe the free 2-rig on $N$ bosonic sublines in two different ways. The proof would be quick if we knew
\[  \Rep(\C \times \D) \simeq \Rep(\C) \boxtimes \Rep(\D) \]
for all affine categories $\C$ and $\D$, or even just all affine monoids. So far we have only shown this for affine monoids arising from commutative monoids via the recipe in \cref{ex:commutative_monoids_as_affine_monoids}. Luckily this is all we need.

\begin{lem}
\label{lem:rep_of_product}
For any commutative monoids $M, N$ there is an equivalence of 2-rigs
\[  \Rep(\Spec(kM) \times \Spec(kN)) \simeq  \Rep(\Spec(kM)) \boxtimes \Rep(\Spec(kN)) .\]
\end{lem}

\begin{proof}
By \cref{lem:graded_over_commutative_monoid} we know $\Rep(\Spec(kM)) \simeq \overline{kM}$ for any commutative monoid $M$, so it suffices to show
\[  \overline{k(M \times N)} \simeq  \overline{kM} \boxtimes \overline{kN}   . \]
 We establish this by showing that both sides have the same universal property. 
By \cref{lem:graded_over_commutative_monoid}, 2-rig maps of the form $\overline{k(M \times N)} \to \R$ are equivalent to strong symmetric monoidal functors $M \times N \to \R$, where we identify $M$ and $N$ with discrete symmetric monoidal categories. Since $M \times N$ is the coproduct of $M$ and $N$ in the 2-category of symmetric monoidal categories, such symmetric monoidal functors $M \times N \to \R$ are equivalent to pairs of symmetric monoidal functors $M \to \R$, $N \to \R$. These in turn are equivalent to pairs of 2-rig maps 
\[\Rep(\Spec(kM)) \to \R, \qquad \Rep(\Spec(kN)) \to \R,\] 
again by \cref{lem:graded_over_commutative_monoid}. Finally, such pairs are equivalent to 2-rig maps
\[  \Rep(\Spec(kM)) \boxtimes \Rep(\Spec(kN)) \to \R \]
because $\boxtimes$ is the coproduct of 2-rigs. 
\end{proof}

\begin{thm}
\label{thm:free_2-rig_on_N_bosonic_sublines}
The free 2-rig on $N$ bosonic subline objects is $\Rep(k^N) \simeq \A^{\boxtimes N}$, or equivalently the 2-rig of $\N^N$-graded vector spaces of finite total dimension, with the symmetry defined using the usual symmetry on $\Vect$ in each homogeneous component.
\end{thm}

\begin{proof}
Since the affine monoid $k$ is $\Spec(k\N)$, \cref{lem:rep_of_product} implies that
\[  \Rep(k^N) \simeq \Rep(k)^{\boxtimes N} \]
as 2-rigs. \cref{lem:free_2-rig_on_bosonic_subline} implies that 
\[  \Rep(k)^{\boxtimes N}  \simeq \A^{\boxtimes N} \]
where $\A$ is the free 2-rig on one bosonic subline object. \cref{lem:free_2-rig_on_N_bosonic_sublines} says that $\A^{\boxtimes N}$ is the free 2-rig on $N$ bosonic subline objects. By \cref{lem:graded_over_commutative_monoid}, $\Rep(k^N)$ is also the 2-rig of finite-dimensional $\N^N$-graded vector spaces, with the symmetry defined using the usual symmetry on $\Vect$ in each homogeneous component.
\end{proof}

\section{The free 2-rig on several bosonic line objects}
\label{sec:N_bosonic_line_objects}

Just as we can describe the free 2-rig on $N$ bosonic subline objects in two slightly different ways, we can do the same for free 2-rig on $N$ bosonic line objects. The only difference is that throughout the discussion the multiplicative monoid $k$ is replaced by its submonoid $k^\ast$ consisting of nonzero elements, and the 2-rig $\A$ of finite-dimensional $\N$-graded vector spaces is replaced by the 2-rig $\T$ of finite-dimensional $\Z$-graded vector spaces, as introduced in \cref{sec:free_on_bosonic_line}.

We begin with a new description of the 2-rig of algebraic representations of the affine group $k^\ast \cong \Spec(k\Z)$ introduced in \cref{ex:GL_1}.

\begin{lem}
\label{lem:free_2-rig_on_bosonic_line}
The 2-rig $\Rep(k^{\ast})$ is the free 2-rig on a bosonic line object, $\T$.
\end{lem}

\begin{proof}
By \cref{lem:graded_over_commutative_monoid}, $\Rep(k^\ast) \simeq \Rep(\Spec(k\mathbb{Z}))$ is the 2-rig of $\Z$-graded vector spaces of finite total dimension, or more precisely $\A$, which according to \cref{thm:free_2-rig_on_bosonic_subline} is the free 2-rig on a bosonic subline object.
\end{proof}

\begin{thm}
\label{thm:free_2-rig_on_N_bosonic_lines}
 The free 2-rig on $N$ bosonic line objects is $\Rep({k^\ast}^N) \simeq \T^{\boxtimes N}$, or equivalently the 2-rig of $\Z^N$-graded vector spaces of finite total dimension, with the symmetry defined using the usual symmetry on $\Vect$ in each homogeneous component.
\end{thm}

\begin{proof}
The proof follows the same argument as that for \cref{thm:free_2-rig_on_N_bosonic_sublines}, but with $k^\ast$ replacing $k$, $\Z$ replacing $\N$ and $\T$ replacing $\A$.
\end{proof}

\section{Dimension and subdimension}

A line object can be thought of as having dimension 1, and a subline object as having dimension at most 1. In fact, these are special cases of more general concepts: we can say what it means for an object in a 2-rig having dimension $d$, or dimension at most $d$. These again come in `bosonic' and `fermionic' forms.

The symmetric group $S_n$ has two one-dimensional representations: the trivial representation, which we call $\triv$, and the sign representation, where each permutation $\sigma \in S_n$ acts as multiplication by its sign $\sign(\sigma)$. These give Schur functors, which act on any 2-rig $\R$ sending any object $x \in \R$ to its $n$th symmetric power
\[   \Sym^n(x) = \triv \otimes_{k[S_n]} x^{\otimes n} \]
and its $n$th exterior power
\[   \Lambda^n(x) = \sign \otimes_{k[S_n]} x^{\otimes n} \] 
respectively.

\begin{defn} 
An object $x$ in a 2-rig has \define{bosonic subdimension $n$} if $\Lambda^{n+1}(x) \cong 0$, and \define{fermionic subdimension $n$} if $\Sym^{n+1}(x) \cong 0$.
\end{defn}

\begin{defn} 
An object $x$ in a 2-rig has \define{bosonic dimension $n$} if $\Lambda^n(x)$ is a bosonic line object, and \define{fermionic dimension $n$} if $\Sym^n(x)$ is a fermionic line object.
\end{defn}

Note that an object has bosonic (resp.\ fermionic) subdimension 1 if and only if it is a bosonic (resp.\ fermionic) subline object, and it has bosonic (resp.\ fermionic) dimension 1 if and only if it is a bosonic (resp.\ fermionic) line object.

\begin{lem}
\label{lem:subdim_inequality}
If an object $x$ in a 2-rig has bosonic (resp.\ fermionic) subdimension $n$, it has bosonic (resp.\ fermionic) subdimension $m$ for all $m \ge n$.
\end{lem}

\begin{proof}
For the first, note that the canonical epimorphism
$x^{\otimes n+1} \to \Lambda^{n+1} (x)$
factors through $x \otimes \Lambda^n(x)$, so 
\[  \Lambda^n (x) \cong 0 \; \implies \; 
\Lambda^{n+1} (x) \cong 0 \]
and thus an object of subdimension $n$ has subdimension $n+1$. For the second, similarly note that the canonical epimorphism $x^{\otimes n+1} \to \Sym^{n+1} (x)$ factors through $x \otimes \Sym^n(x)$. 
\end{proof}

\begin{lem}
If $x$ and $y$ are objects in a 2-rig with bosonic subdimensions $m$ and $n$, respectively, then $x \oplus y$ has bosonic subdimension $m+n$. Similarly if $x$ and $y$ have fermionic subdimension $m$ and $n$, respectively, then $x \oplus y$ has fermionic subdimension $m+n$.
\end{lem}

\begin{proof} 
Assume $x$ has bosonic subdimension $m$ and $y$ has bosonic subdimension $n$. By \cref{lem:subdim_inequality} we have $\Lambda^m(x) \cong 0$ for $M > m$ and $\Lambda^n(y) \cong 0$ for $N > n$. Using the well-known isomorphism
\[   \Lambda^k(x \oplus y) \cong \bigoplus_{m+n=k}
\Lambda^m(x) \otimes \Lambda^n(x) \]
it follows that $\Lambda^k(x \oplus y) \cong 0$ for $k > m + n$, so $x \oplus y$ has bosonic subdimension $m+n$. The same argument works for fermionic subdimension using the isomorphism
\[   \Sym^k(x \oplus y) \cong \bigoplus_{m+n=k}
\Sym^m(x) \otimes \Sym^n(x) . \qedhere \]
\end{proof} 

\begin{cor}
\label{cor:bosonic_subline_coproduct}
If $s_1, \ldots, s_N$ are bosonic sublines, then the coproduct $s_1 \oplus \cdots \oplus s_N$ is of bosonic subdimension $N$. 
\end{cor}

The following trio of conjectures, if true, would clarify the overall picture laid out in the next section. However, we do not strictly need them in what follows.

\begin{conj} 
\label{conj:dimension_vs_subdimension}
If an object in a 2-rig has bosonic dimension $n$, it has bosonic subdimension $n$. If it has fermionic dimension $n$, it has fermionic subdimension $n$. 
\end{conj}

\begin{conj}
\label{conj:free_2-rig_on_object_of_subdimension_N} 
The 2-rig $\Rep(\M(N,k))$ is the free 2-rig on an object of bosonic subdimension $N$, namely the tautologous representation of $\M(N,k)$ on $k^N$. 
\end{conj}

\begin{conj}
\label{conj:free_2-rig_on_object_of_dimension_N} 
The 2-rig $\Rep(\GL(N,k))$ is the free 2-rig on an object of bosonic dimension $N$, namely the tautologous representation of $\GL(N,k)$ on $k^N$. 
\end{conj}

\section{A network of 2-rigs}
\label{sec:network}

Our main results concern a number of 2-rigs important in representation theory, and maps between these:
\[
\begin{tikzcd}
    \ksbar
    \arrow[r,"A"]
    & 
    \Rep(\M(N,k)) 
    \ar[r,"B"]
    \ar[d,"D"']
    &
    \Rep(k^N) \simeq \A^{\boxtimes N}
    \arrow[d,"C"]
    \\
    & 
    \Rep(\GL(N,k))
    \ar[r,"E"']
    & 
    \Rep({k^\ast}^N) \simeq \T^{\boxtimes N}
\end{tikzcd}
\]

The objects in this diagram are as follows:
\begin{itemize}
\item $\ksbar$ is the category of finitely supported and finite-dimensional representations of the groupoid of finite sets, $\S$, with its Day tensor product. In \cite[Thm.\ 3.3]{Schur} we proved that $\ksbar$ is the free 2-rig on one object.
\item $\Rep(\M(N,k))$ is the 2-rig of algebraic representations of $\M(N,k)$, the affine monoid of $N \times N$ matrices with entries in $k$. \cref{conj:free_2-rig_on_object_of_subdimension_N} claims that $\Rep(\M(N,k))$ is also the free 2-rig on an object of bosonic subdimension $N$, namely $k^N$.
\item $\Rep(k^N)$ is the 2-rig of algebraic representations of $k^N$, which becomes an affine monoid under pointwise multiplication. We described this 2-rig in several different ways in \cref{thm:free_2-rig_on_N_bosonic_sublines}.  It is the free 2-rig on $N$ bosonic subline objects $s_1, \dots, s_N$.  It is also the $N$-fold tensor power $\T^{\boxtimes N}$ of the free 2-rig on one bosonic subline object, and the 2-rig of $\N^N$-graded vector spaces of finite total dimension.
\item $\Rep(\GL(N,k))$ is the 2-rig of algebraic representations of $\GL(N,k)$, the affine group of invertible $N \times N$ matrices with entries in $k$. \cref{conj:free_2-rig_on_object_of_dimension_N} claims that $\Rep(\GL(N,k))$ is also the free 2-rig on an object of bosonic dimension $N$, namely $k^N$.
\item $\Rep({k^\ast}^N)$ is the 2-rig of algebraic representations of ${k^*}^n$, which becomes an affine group under pointwise multiplication. We described this 2-rig in several ways in \cref{thm:free_2-rig_on_N_bosonic_lines}.  It is the free 2-rig on $N$ bosonic line objects $\ell_1, \dots, \ell_N$.  It is also the $N$-fold tensor power $\A^{\boxtimes N}$ of the free 2-rig on one bosonic line object, and the 2-rig of $\Z^N$-graded vector spaces of finite total dimension.
\end{itemize}
The arrows in this diagram can be defined using representation theory:
\begin{itemize}
\item The 2-rig map $A \maps \ksbar \to \Rep(\M(N,k))$ sends the generator $x \in \ksbar$ to the tautologous representation of the affine monoid $\M(N,k)$ on the vector space $k^N$. 
\item The 2-rig map $B \maps \Rep(\M(N,k)) \to \Rep(k^N)$ is given by restricting algebraic representations of $M(n)$ to the submonoid consisting of diagonal matrices.
\item The 2-rig map $C \maps \Rep(k^N)  \to \Rep({k^\ast}^N)$ restricts algebraic representations of $k^N$ to the subgroup ${k^\ast}^N$.
\item The 2-rig map $D \maps \Rep(\M(N,k)) \to \Rep(\GL(N,k))$ restricts algebraic representations of $\M(N,k)$ to $\GL(N,k)$.
\item The 2-rig map $E \maps \Rep(\GL(N,k)) \to \Rep({k^*}^N)$ restricts representations of $\GL(N,k)$ to the subgroup consisting of invertible diagonal matrices.
\end{itemize}
We expect that all these arrows can also be defined using the universal properties of the various 2-rigs involved:
\begin{itemize}
\item The 2-rig map $A \maps \ksbar \to \Rep(\M(N,k))$ sends the generator $x$ of the free 2-rig on one generator to the object $k^N$. 
\item Given \cref{conj:free_2-rig_on_object_of_subdimension_N}, the 2-rig map $B \maps \Rep(\M(N,k)) \to \Rep(k^N)$ sends the generator of the free 2-rig on an object of bosonic subdimension $n$, namely $k^N$, to the direct sum of bosonic subline objects $s_1 \oplus \cdots \oplus s_n$.
\item The 2-rig map $C \maps \Rep(k^N)  \to \Rep({k^\ast}^N)$ sends each bosonic subline object $s_i$ to the bosonic line object $\ell_i$.
\item Given Conjectures \ref{conj:dimension_vs_subdimension}--\ref{conj:free_2-rig_on_object_of_dimension_N}, the 2-rig map $D \maps \Rep(\M(N,k)) \to \Rep(\GL(N,k))$ maps the generator of the free 2-rig on an object of bosonic subdimension $n$ to the generator of the free 2-rig on an object of bosonic dimension $n$.
\item Given \cref{conj:free_2-rig_on_object_of_dimension_N}, the 2-rig map $E \maps \Rep(\GL(N,k)) \to \Rep({k^*}^n)$ maps the generator of the free 2-rig on an object of bosonic dimension $N$, namely $k^N$, to the direct sum of bosonic line objects $\ell_1 \oplus \cdots \oplus \ell_N$.
\end{itemize} 

The technical heart of this paper is to prove that all these maps are `extensions', in the following sense:

\begin{defn}
\label{defn:extension}
    A functor $F \maps \C \to \D$ is an \define{extension} if the underlying functor has the following three properties: 
\begin{itemize}
 \item {\rm (Fa)} $F$ is faithful: if $f, g \maps c \to c'$ are morphisms in $\R$ such that  $F(f) = F(g)$, then $f = g$. 
 
 \item {\rm (Co)} $F$ is conservative: if  $F(f)$ is an isomorphism, then $f$ is an isomorphism. 
 
 \item {\rm (Es)} $F$ is essentially injective: if $c,c'$ are objects of $\C$ such that  $F(c) \cong F(c')$, then $c \cong c'$.
\end{itemize}
If $F \maps \R \to \S$ is a 2-rig map and also an extension, we call it a \define{2-rig extension}.
\end{defn}
\noindent
All three properties above can be seen as forms of `injectivity'. The third is generally the hardest to check. For all three, we will repeatedly use the following easy lemma:

\begin{lem} 
\label{lem:extension}
Let $G \maps \C \to \D$ and $F \maps \D \to \E$ be functors.
Then:
\begin{itemize}
\item If $F$ and $G$ both satisfy one of the conditions {\rm (Fa), (Co)} or {\rm (Es)}, then $F \circ G$ satisfies that condition.
\item If $F \circ G$ satisfies one of {\rm (Fa), (Co)} or {\rm (Es)}, then $G$ satisfies that condition.
\item If $F \circ G$ satisfies one of {\rm (Fa)}, {\rm (Co)} and $G$ is full, then $F$ satisfies that condition. 
\item If $F \circ G$ satisfies {\rm (Es)} and $G$ is essentially surjective, then $F$ satisfies {\rm (Es)}.
\end{itemize}
\end{lem} 

Another easy lemma concerns the case where $\C$ is a $k$-linear semisimple category: 

\begin{lem}
\label{lem:semisimple_extension}
If $\C, \D$ are $k$-linear categories and $\C$ is semisimple, then a $k$-linear functor $F \maps \C \to \D$ is faithful if and only if for any map $f \maps c \to c'$ between simple objects in $\C$, the condition $F(f) = 0$ implies $f = 0$. The functor $F$ is conservative if and only if for any $f \maps c \to c'$ between simple objects, $f$ is invertible whenever $F(f)$ is invertible. 
\end{lem}

\section{Splitting an object of finite dimension}

As a prelude to more general `splitting principles' we study the process of splitting an object of bosonic dimension $N$ into $N$ bosonic line objects, and show that in one key case this process gives an extension of 2-rigs. 

For any field $k$ of characteristic zero and any natural number $N$ there is a map
\[   
\begin{array}{rccc}
j \maps & {k^\ast}^N & \to & \GL(N,k)  \\ \\
& (x_1, \dots, x_N) & \mapsto & 
\left( \begin{array}{cccc} 
x_1 & 0 & \cdots & 0 \\
0 & x_2 & \cdots & 0 \\
\vdots & \vdots & \ddots &  \vdots \\
0 & 0 &  \cdots & x_N \end{array}
 \right)
\end{array}
\]
This map is an algebraic homomorphism between affine groups, so by \cref{cor:Rep_as_a_2-functor}, restricting representations along it induces a 2-rig map
\[   E = \Rep(j) \maps \Rep(\GL(N,k)) \to \Rep({k^\ast}^N). \]
This 2-rig maps sends the representation $k^N$ to the sum $s_1 \oplus \cdots \oplus s_N$. So, we can say it splits $k^N$ into a sum of $N$ bosonic line objects.

We now show that $E$ is a 2-rig extension in the sense of \cref{defn:extension}. This fact can be proved using Young tableaux for instance. More in the spirit of this paper, this fact is also important in the theory of algebraic groups, since $\GL(N,k)$ is perhaps the most fundamental example of a `split reductive' algebraic group, and ${k^\ast}^N$ is its `maximal torus'. The algebraic representations of any split reductive algebraic group are determined up to isomorphism by their restriction to its maximal torus, and this gives the essential injectivity of $E$. We spell this out below.

\begin{lem}
\label{lem:E_ess_inj}
The 2-rig map $E \maps \Rep(\GL(N,k)) \to \Rep({k^\ast}^N)$ is a 2-rig extension.
\end{lem}

\begin{proof} 
That $E$ is faithful and conservative follows from \cref{lem:extension} because we have a commutative diagram
\[
\begin{tikzcd}[column sep=0em]
    \Rep(\M(N,k)) 
    \ar[rr,"E"]
    \ar[dr]
    & &
    \Rep(\GL(N,k))
    \ar[dl]
     \\
   & \Set. 
\end{tikzcd}
\]
where the downwards arrows, the obvious forgetful functors to $\Set$, are both faithful and conservative.

To prove that $E$ is essentially injective, recall that $E$ pulls back representations along the inclusion of algebraic groups
\[    j \maps {k^\ast}^N \to \GL(N,k) .\]
In characteristic zero, an algebraic group $G$ is reductive if and only if it is connected and $\Rep(G)$ is semisimple \cite[Thm.\ 22.42]{MilneAlgGp}. 
In fact the algebraic group $\GL(N,k)$ together with its subgroup $\im(j) \cong {k^\ast}^N$ is split reductive \cite[Ex. 21.6]{MilneAlgGp}. This subsumes the fact that $\GL(N,k)$ is reductive, and it also implies \cite[Thm.\ 22.48]{MilneAlgGp} that $E$ induces an injection of Grothendieck groups
\[     K(\Rep(\GL(N,k))) \to K(\Rep({k^\ast}^N)). \]

Now, suppose that $a,b \in \Rep(\GL(N,k))$ have $E(a) \cong E(b)$. It follows that $[E(a)] = [E(b)]$ in $K(\Rep({k^\ast}^N))$, so we must have $[a] = [b]$ in $K(\Rep(\GL(N,k)))$. This means that $a \oplus c \cong b \oplus c$ for some $c \in \Rep(\GL(N,k))$, but since 
$\Rep(\GL(N,k))$ is semisimple this implies $a \cong b$. Thus $E$ is essentially injective. \end{proof}

The use of results on reductive algebraic groups would be disappointing if one were hoping for a more elementary proof of \cref{lem:E_ess_inj}. We know of no such proof. We have provided detailed references to Milne's textbook because he proves everything from scratch, but the path to what we use above is not a short one. The detour into Grothendieck groups, at least, can be shortcut: Milne really proves the essential injectivity of $E$, but he only states the injectivity of $K(E)$, so to extract the result we need, we used the semisimplicity of $\Rep(\GL(n,k))$. 

However, we really only need \cref{lem:E_ess_inj} in the special case $k = \CC$; see the discussion following \cref{lem:BAi_ess_inj}. Here the essential injectivity of $E$ can be proved more easily, using Lie theory. Since this is the case relevant to the classical splitting principle for complex vector bundles, we sketch the argument here. 

The key is Weyl's `unitarian trick', which exploits this commutative diagram of 2-rigs:
 \[
\begin{tikzcd}
    \Rep(\GL(N,\CC)) 
    \arrow[r, "F"]
    \arrow[d, "E"']
    &
    \Rep(\U(N))
    \arrow[d, "G"]
    \\
    \Rep({\CC^*}^N) \arrow[r, "H"'] &
    \Rep(\U(1)^N)
\end{tikzcd}
\]
Here $\U(N)$ is the subgroup of $\GL(N,\CC)$ consisting of unitary matrices, ${\CC^*}^n$ is the subgroup consisting of invertible diagonal matrices and $\U(1)^N = \U(N) \cap {\CC^*}^N$. Here $\Rep$ denotes the category of complex-algebraic representations for $\GL(N,\CC)$ and ${\CC^*}^N$, as usual in this paper, but for $\U(N)$ and $\U(1)^N$ it stands for the category of continuous finite-dimensional complex representations. All the maps in the diagram arise from restricting representations.

An algebraic representation $\rho \maps \GL(N,\CC) \to \End(V)$ is determined by its restriction to $\U(N)$ since $\U(N)$ is Zariski dense in $\GL(N,\CC)$. A continuous finite-dimensional representation $\rho \maps \U(N) \to \End(V)$ is determined up to isomorphism by its character since $\U(N)$ is a compact Lie group, and its character is determined by its restriction to the diagonal subgroup $\U(1)^N$ since characters are class functions and every unitary matrix is conjugate to a diagonal one. Thus, a representation $\rho \in \Rep(\GL(N,k))$ is determined up to isomorphism by $G(F(\rho)) \in \Rep(\U(1)^N)$. In other words, the composite $G \circ F$ is essentially injective. This implies that $H \circ E$ and thus $E$ is essentially injective.

\section{Splitting an object of finite subdimension}
\label{sec:splitting_finite_subdimension}

Next we study the process of splitting an object of subdimension $N$ into $N$ subline objects. We do this in one key case, namely the object $k^N \in \Rep(\M(N,k))$, and we
show that the 2-rig map 
\[   B \maps \Rep(\M(N,k)) \to \Rep(k^N)  \]
sending $k^N$ to $s_1 \oplus \cdots \oplus s_N$ is an extension of 2-rigs.

To prove this fact, we recall a portion of our main diagram:
\[
\begin{tikzcd}
    \Rep(\M(N,k)) 
    \ar[r,"B"]
    \ar[d,"D"']
    &
    \Rep(k^N) 
    \arrow[d,"C"]
    \\
    \Rep(\GL(N,k))
    \ar[r,"E"']
    & 
    \Rep({k^\ast}^N). 
\end{tikzcd}
\]
We have already seen that $E$ is an extension; we now prove this for the two vertical arrows and then $B$. 

\begin{lem}
\label{lem:C_ess_inj}
The 2-rig map $C \maps \Rep(k^N) \to \Rep({k^\ast}^N)$ is a 2-rig extension.
\end{lem}

\begin{proof} 
This follows from Theorems \ref{thm:free_2-rig_on_N_bosonic_sublines} and \ref{thm:free_2-rig_on_N_bosonic_lines}: $\Rep(k^N)$ is the category of finite-dimensional $\N^N$-graded vector spaces, $\Rep({k^\ast}^N)$ is the category of finite-dimensional $\Z^N$-graded vector spaces, and $B$ is the forgetful functor from the former to the latter. This functor is not only faithful but full, so it is conservative and essentially injective.
\end{proof}

\begin{lem}
\label{lem:D_ess_inj}
The 2-rig map $D \maps \Rep(\M(N,k)) \to \Rep(\GL(N,k))$ is a 2-rig extension.
\end{lem}

\begin{proof}
That $D$ is faithful and conservative follows from same argument as in Lemma \ref{lem:E_ess_inj}. For essential injectivity, suppose that $D(\rho) \cong D(\rho')$ where $\rho \maps \M(N,k) \to \End(V)$, $\rho \maps \M(N,k) \to \End(V')$ are algebraic representations. Concretely this means that for some isomorphism $\phi \maps V \to V'$ we have $\rho'(g) = \phi \rho(g) \phi^{-1}$ for all $g \in \GL(N,k)$. Since $\GL(N,k)$ is Zariski dense in $\M(N,k)$ the same equation holds for all $g \in \M(N,k)$, so $\rho' \cong \rho$.
\end{proof}

\begin{lem}
\label{lem:B_ess_inj}
The 2-rig map $B \maps \Rep(\M(N,k)) \to \Rep(k^N)$ is a 2-rig extension.
\end{lem}

\begin{proof}
By \cref{lem:extension} it suffices to show the composite $C \circ B$ is a 2-rig extension. Since this composite is naturally isomorphic to $E \circ D$, this follows from the fact that $D$ and $E$ are 2-rig extensions (Lemmas \ref{lem:D_ess_inj} and \ref{lem:E_ess_inj}).
\end{proof}

We conclude this sections with some results, relying on those above, that we shall use to prove the splitting principle. For these we slightly extend our main diagram:
\[
\begin{tikzcd}
    \ksbar_{\le n} 
      \arrow[r,"i"]
    &
    \ksbar
    \arrow[r,"A"]
    & 
    \Rep(\M(N,k)) 
    \ar[r,"B"]
    \ar[d,"D"']
    &
    \Rep(k^N) 
    \arrow[d,"C"]
    \\
    & &
    \Rep(\GL(N,k))
    \ar[r,"E"']
    & 
    \Rep({k^\ast}^N) 
\end{tikzcd}
\]
Here $\ksbar_{\le n}$ is the category of finite-dimensional representations of the groupoid of sets with at most $n$ elements, which we call $\S_{\le n}$. The functor
\[i \maps \ksbar_{\le n} \longrightarrow \ksbar \]
extends any finite-dimensional representation of $\S_{\le n}$ to one that assigns the zero-dimensional space to any set of cardinality $> n$. Note that $i$ is not a 2-rig map, just a $k$-linear functor between Cauchy complete linear categories. 

\begin{thm}
\label{thm:Ai_ess_inj}
The following composite functor is an extension when $n \le N$:
\[ \begin{tikzcd}
    \overline{k\S}_{\le n} 
      \arrow[r,"i"]
    &
    \overline{k\S}
    \arrow[r,"A"]
    & 
    \Rep(\M(N,k)).
\end{tikzcd}
\]
\end{thm} 

The proof will be broken up into a series of lemmas.  First we introduce $\overline{k\S}_m$, which is the category of finite-dimensional representations of the groupoid of $m$-element sets---or equivalently, the category of finite-dimensional representations of the group $S_m$.  Whenever $m \le n$, we have an inclusion of Cauchy complete linear categories
\[          \overline{k\S}_m \longrightarrow \overline{k\S}_{\le n} \]
extending any such representation to one that assigns the zero-dimensional space to any set of cardinality $\ne m$.

\begin{lem}
Suppose that the restriction of $A \circ i \maps \ksbar_{\le n} \to \Rep(\M(N, k))$ to $\ksbar_m$ is an extension when $m \le n \le N$.  Then $A \circ i$ is an extension when $n \le N$.
\end{lem} 

\begin{proof}
\cref{lem:semisimple_extension} implies that if for each
$m \le n \le N$ the restriction of $A \circ i$ to $\ksbar_m$ is faithful and conservative, then the same holds for $A \circ i$.  Thus we only need to show this: if for each $m \le n \le N$ the restriction of $A \circ i$ to $\ksbar_m$ is essentially injective, then the same holds for $A \circ i$.

Let $\rho, \pi$ be objects of $\overline{k\S}_{\le n}$. Write $\rho = \rho_0 \oplus \cdots \oplus \rho_n$ where $\rho_j$ is an $S_j$-representation for $1 \le j \le n$, and similarly write $\pi = \pi_0 \oplus \cdots \oplus \pi_n$. Let $V = k^N$ be the tautological $\M(N, k)$-representation. Then 
\[
(A \circ i)(\rho) = \bigoplus_{1 \le j \le n} \rho_j \otimes_{kS_j} V^{\otimes j}
\]
with a similar expression for $(A \circ i)(\pi)$. Abbreviating the summands in the display above as $\widetilde{\rho_j}(V)$, suppose we have an $\M(N, k)$-equivariant map 
\[
    f \maps \widetilde{\rho_0}(V) \oplus \cdots \oplus \widetilde{\rho_n}(V) \to \widetilde{\pi_0}(V) \oplus \cdots \oplus \widetilde{\pi_n}(V),
\]
and denote the block components as $f_{ij} \maps \widetilde{\rho_i}(V) \to \widetilde{\pi_j}(V)$. We now show that $f_{ij} = 0$ if $i \neq j$.  For a scalar $\lambda \in k$, abuse notation slightly by letting $\lambda \in \M(N, k)$ also denote the corresponding scalar multiple of the identity on $V$. Then 
\[
\widetilde{\rho_i}(\lambda) = \lambda^i \maps \widetilde{\rho_i}(V) \to \widetilde{\rho_i}(V)
\]
and similarly $\widetilde{\pi_j}(\lambda) = \lambda^j$. By equivariance, $\lambda^j f_{ij} = f_{ij} \lambda^i = \lambda^i f_{ij}$. For $i \neq j$, this forces $f_{ij} = 0$. 

Thus, an equivariant map $f$ must be in block diagonal form $f_{11} \oplus \cdots \oplus f_{nn}$. As a result, $f$ is an isomorphism if and only if each of its block components $f_{mm}$ is an isomorphism. Thus, if the restriction of $A \circ i$ to $\ksbar_m$ is essentially injective for
each $m \le n$, then the same holds for $A \circ i$.
\end{proof}

Thanks to the preceding lemma, to prove \cref{thm:Ai_ess_inj} it suffices to show that the restriction of $A \circ i$ to $\ksbar_m$ is an extension whenever $m \le n \le N$.  This amounts to showing that under these conditions, the functor sending $\rho \in \ksbar_m$ to 
\[  \tilde{\rho}(V) = \rho_m \otimes_{kS_j} V^{\otimes m} \in [\M(N,k),\Vect] \]
is an extension.   In fact we prove a stronger result: this functor is fully faithful.   We need to show this statement whenever $m \le n \le N$, but $n$ does not appear in this statement so there is no loss of generality in taking $n = m$, which we do below.

Let $[\M(N, k), \Vect]$ denote the category of all linear representations of the monoid $\M(N, k)$.  The forgetful functor
\[
\Rep(\M(N, k)) \to [\M(N, k), \Vect]
\]
is clearly faithful, which leads to the following lemma.

\begin{lem}
If the composite functor 
    \[
    \ksbar_n \to \Rep(\M(N, k)) \to [\M(N, k), \Vect]
    \]
    is fully faithful, then $\ksbar_n \to \Rep(\M(N, k))$ is fully faithful. 
\end{lem} 

\begin{proof}
In any category, $k \maps A \to B$ is an isomorphism if $h \circ k \maps A \to C$ is an iso for some mono $h \maps B \to C$.  The result follows by applying this observation to the evident maps between hom-sets, where full faithfulness corresponds to such maps being isos, and faithfulness corresponds to such maps being monos.  
\end{proof} 

By this last lemma, the proof of \cref{thm:Ai_ess_inj} has now been reduced to showing the functor $\ksbar_n \to [\M(N, k), \Vect]$ is fully faithful. The category $\ksbar_n$ is the Cauchy completion of the one-object linear category $kS_n$. Here, the endo-hom of this object is the regular representation of $S_n$, and applying the functor $\rho \mapsto \widetilde{\rho}(V)$ to this representation we obtain 
\[
kS_n \otimes_{kS_n} V^{\otimes n} \cong V^{\otimes n}.
\]

\begin{lem}
If the restriction of $\ksbar_n \to [\M(N, k), \Vect]$ to $kS_n$ is fully faithful, i.e., if the canonical map 
\[
   R \maps kS_n \to [\M(N, k), \Vect](V^{\otimes n}, V^{\otimes n})
\]
is an isomorphism, then $\overline{kS_n} \to [\M(N, k), \Vect]$ is fully faithful. 
\end{lem}

\begin{proof}
More generally, if $F \maps \C \to \D$ is a fully faithful $k$-linear functor and $\D$ is Cauchy complete, then the functor $F' \maps \overline{\C} \to \D$ obtained by extending $F$ to the Cauchy completion $\overline{\C}$ is also fully faithful. Indeed, this is a general fact in enriched category theory.  To prove it in the $k$-linear case, first note that  we have a square of enriched functors, commuting up to isomorphism
    \[
    \begin{tikzcd}
        \overline{\C} 
        \ar[d, swap, "{F'}"] 
        \ar[r, "y_C"] 
        & 
        {[\overline{\C}\op, \Vect]} 
        \ar[r, "\sim"] 
        & 
        {[\C\op, \Vect]} 
       \ar[d, "PF"]
        \\ 
        \D 
        \ar[rr, "y_D"] 
        & &
        {[\D\op, \Vect]}
    \end{tikzcd}
    \]
of functors, commuting up to isomorphism, where $PF$ is obtained by left Kan extending the composite 
    \[
    \C \xlongrightarrow{F} \D \xlongrightarrow{y_D} [\D\op, \Vect]
    \]
along the Yoneda embedding $y_C \maps \C \to [\C\op, \Vect]$.  In the above square the horizontal arrows are fully faithful by the enriched Yoneda lemma.  Thus, if we can show $PF$ is fully faithful, we can conclude that $F'$ is fully faithful as well.

Using the adjunction $PF \dashv F^\ast = [F\op, \Vect]$,  the theory of  coreflective enriched subcategories implies that $PF$ is fully faithful if the unit of this adjunction is an isomorphism.  For any $W \in [\C\op,\Vect]$ the unit acts as follows:
    \[
    W \cong \int^c Wc \cdot \C(-, c) \to \int^c Wc \cdot D(F-, Fc) \cong (F^\ast \circ PF) W.
    \]
The middle arrow is an isomorphism by the full faithfulness of $F$, 
so the unit is indeed an isomorphism.    \end{proof}

The proof of \cref{thm:Ai_ess_inj} is therefore complete once we prove our final lemma: 

\begin{lem}
     If $V = k^N$, the canonical map 
    \[
    kS_n \to [\M(N, k), \Vect](V^{\otimes n}, V^{\otimes n})
    \]
    is an isomorphism. 
\end{lem}

\begin{proof}
    Let $e_1, \ldots, e_N$ be the standard basis of $V$. There are $N^n$ basis elements for $V^{\otimes n}$, indexed by functions $f\maps [n] \to [N]$ where $[n] \coloneqq \{1, \ldots, n\}$. Thus a typical basis element is $e_{f(1)} \otimes \cdots \otimes e_{f(n)}$, and a typical element is uniquely representable as a $k$-linear combination 
\[
\sum_{f\maps [n] \to [N]} a_f \, e_{f(1)} \otimes \cdots \otimes e_{f(n)}
\]
Let $g = e_1 \otimes \cdots \otimes e_n$ (a `generic element' of $V^{\otimes n}$). An element $\sigma \in S_n$ is uniquely determined by its action on $g$, since of course $\sigma \cdot g = e_{\sigma(1)} \otimes \cdots \otimes e_{\sigma(n)}$. Thus, the map
\[
kS_n \to [\M(N, k), \Vect](V^{\otimes n}, V^{\otimes n})
\]
is injective. It remains to prove the claim that it is surjective. 

First, notice that if $\psi \maps V^{\otimes n} \to V^{\otimes n}$ is a $\M(N, k)$-equivariant map, where $\M(N, k)$ acts diagonally on $V^{\otimes n}$ by
\[
L \cdot (v_1 \otimes \cdots \otimes v_n) = L^{\otimes n}(v_1 \otimes \cdots \otimes v_n) = L(v_1) \otimes \cdots \otimes L(v_n),
\]
then $\psi$ is uniquely determined by the value $\psi(g)$. Indeed, if $e_{f(1)} \otimes \ldots \otimes e_{f(n)}$ is any basis element, then taking $L \in \M(N, k)$ to be any linear map such that $L(e_i) = e_{f(i)}$ for $1 \leq i \leq n$, equivariance forces the equation
\[
\psi(e_{f(1)} \otimes \cdots \otimes e_{f(n)}) = L \cdot \psi(g).
\]
Given such a map $\psi$, write the element $\psi(g)$ as
\[
\psi(g) = \sum_f a_f e_{f(1)} \otimes \cdots \otimes e_{f(n)},
\]
summing over functions $f \maps [n] \to [N]$. The claimed surjectivity will follow if $a_f = 0$ unless $f$ is a permutation on $[n]$, for in that case we have
\[
\psi(g) = \psi(e_1 \otimes \cdots \otimes e_n) = \sum_{\sigma \in S_n} a_\sigma e_{\sigma 1} \otimes \cdots \otimes e_{\sigma n} = (\sum_{\sigma \in S_n} a_\sigma \sigma) \cdot (e_1 \otimes \cdots \otimes e_n)
\]
and thus
\[  \psi(v_1 \otimes \cdots \otimes v_n) = (\sum_{\sigma \in S_n} a_\sigma \sigma) \cdot (v_1 \otimes \cdots \otimes v_n) \]
for all $v_i \in V$, by applying the $M_N(k)$-equivariance of $\psi$ to any linear map $L$ with $L(e_i)=v_i$.  The last displayed equation shows that $\psi$ is indeed in the image of the map $kS_n \to [\M(N, k), \Vect](V^{\otimes n}, V^{\otimes n})$.

To see $a_f = 0$ unless $f(i)\in [n]$ for all $i \in [n]$, consider the linear map $L\maps V \to V$ defined by the rule $L(e_i) = e_i$ if $1 \leq i \leq n$, and $L(e_i)=2e_i$ if $n < i \leq N$. Then
\[
(\psi \circ L^{\otimes n})(e_1 \otimes \cdots \otimes e_n) = \psi(e_1 \otimes \cdots \otimes e_n) = \sum_f a_f e_{f(1)} \otimes \cdots \otimes e_{f(n)}
\]
whereas
\begin{align*}
    (L^{\otimes n} \circ \psi)(e_1 \otimes \cdots \otimes e_n) 
    &= L^{\otimes n}\left(\sum_{f \maps [n] \to [N]} a_f e_{f(1)} \otimes \cdots \otimes e_{f(n)}\right) 
    \\&= \sum_{f \maps [n] \to [N]} 2^{m_f} a_f e_{f(1)} \otimes \cdots \otimes e_{f(n)}
\end{align*}
where $m_f$ is the size of $\{i \leq n|\; f(i)>n\}$. Matching coefficients, we have $a_f = 2^{m_f} a_f$, so that $a_f = 0$ unless $f(i) \leq n$ for all $i \leq n$.  Thus, under equivariance, $\psi(g)$ is of the form $\sum_{f \maps [n] \to [n]} a_f e_{f(1)} \otimes \cdots \otimes e_{f(n)}$. 

To see $a_f = 0$ unless $f$ is surjective, pick any non-surjective $f$, and define a linear map $L$ by the rule $L(e_i) = e_i$ if $i \in \text{im}(f)$, and $L(e_i) = 2e_i$ if $i \notin \text{im}(f)$. Then
\[
(L^{\otimes n} \circ \psi)(g) = L^{\otimes n}\left(\sum_{h \maps [n] \to [n]} a_h e_{h(1)} \otimes \cdots \otimes e_{h(n)}\right)
\]
where the summand at index $f$ is $a_f L(e_{f(1)}) \otimes \cdots \otimes L(e_{f(n)}) = a_f e_{f(1)} \otimes \cdots \otimes e_{f(n)}$. On the other hand, letting $m$ be the size of $[n] \setminus \text{im}(f)$,
\[
(\psi \circ L^{\otimes n})(g) = \psi(L(e_1) \otimes \cdots \otimes L(e_n)) = \psi(2^m e_1 \otimes \cdots \otimes e_n) = \sum_h 2^m a_h e_{h(1)} \otimes \cdots \otimes e_{h(n)}
\]
where the summand at the index $f$ is $2^m a_f e_{f(1)} \otimes \cdots \otimes e_{f(n)}$. These summands agree by equivariance, and this forces $a_f = 0$. This shows that the sum in
\[
\psi(g) = \sum_f a_f e_{f(1)} \otimes \cdots \otimes e_{f(n)}
\]
may be taken over surjective functions $f \maps [n] \to [n]$, in other words permutations $f \in S_n$, and the proof is complete.
\end{proof}

\begin{cor} 
\label{lem:BAi_ess_inj}
The following composite functor is an extension when $N \ge n$:
\[ \begin{tikzcd}
    \overline{k\S}_{\le n} 
      \arrow[r,"i"]
    &
    \overline{k\S}
    \arrow[r,"A"]
    & 
    \Rep(\M(N,k)) 
    \arrow[r,"B"]
    &
    \Rep(k^N) \simeq \A^{\boxtimes N}.
\end{tikzcd}
\]
\end{cor}

\begin{proof} 
Assume $N \ge n$. By \cref{thm:Ai_ess_inj} we know that $A \circ i$ is an extension, and by \cref{lem:B_ess_inj} we know the same for $B$. Thus their composite is an extension.
  \end{proof} 

As noted in the discussion after \cref{lem:E_ess_inj}, we can use ideas from Lie theory to show $B$ is essentially injective in the special case $k = \CC$. Thus, it is interesting that we can prove the essential injectivity $B \circ A \circ i$ for any field $k$ of characteristic zero using the essential injectivity of $B$ only for $k = \CC$. The trick is to use a kind of `up-and-down' argument.

It is well known that $\Q$ is a splitting field for the symmetric groups \cite[Corollary 4.16]{Lorenz}. In our language, this implies that for any field of $k$ of characteristic zero, the functor
\[
\begin{tikzcd}
    \overline{\Q\S}_{\le n} 
      \arrow[r]
    & 
    \overline{k\S}_{\le n}
\end{tikzcd}
\] 
given by tensoring representations with $k$ is both essentially injective and essentially surjective. Clearly, changing coefficients along the field map $\Q \hookrightarrow k$ induces an essentially injective map $\Vect_\Q \to \Vect_k$, and the same is true for graded vector spaces of these types. Taking $k = \CC$ it follows that the two vertical maps in 
\[
\begin{tikzcd}
    \overline{\Q\S}_{\le n} 
      \arrow[rr,"B \circ A \circ i"] \arrow[d]
    && 
    \Rep(\Q^N) 
    \arrow[d] \\
    \overline{\CC\S}_{\le n}
    \arrow[rr,"B \circ A \circ i"'] 
    &&
    \Rep(\CC^N) 
\end{tikzcd}
\]
are essentially injective, as is the bottom horizontal map. It follows from \cref{lem:extension}  that the top horizontal map is essentially injective as well. 

Now let $k$ be any field of characteristic zero. Consider the diagram 
\[
\begin{tikzcd}
    \overline{\Q\S}_{\le n} 
      \arrow[rr, "B \circ A \circ i"] \arrow[d]
    && 
    \Rep(\Q^N) 
    \arrow[d] \\
    \ksbar_{\le n}
    \arrow[rr, "B \circ A \circ i"'] 
    &&
    \Rep(k^N). 
\end{tikzcd}
\]
We have seen that the right vertical map and the top horizontal map are essentially injective. On the other hand, because $\Q$ is a splitting field for the symmetric groups, the left vertical map is essentially {\em surjective}. It follows from \cref{lem:extension} that the bottom horizontal map is essentially injective.

\section{Graded 2-rigs} 
\label{sec:graded_2-rigs}

We are now close to proving the main result of this paper. In \cref{lem:BAi_ess_inj} we constructed an extension of linear categories whenever $N \ge n$:
\[  
    \overline{k\S}_{\le n} \to
    \Rep(k^N) \simeq \A^{\boxtimes N}.
\]
We now wish to study this extension in the limit where $N, n \to \infty$. To do this, we use the inclusions of linear categories
\[   \overline{k\S}_{\le 0} \xrightarrow{\phantom{x}i_0 \phantom{x}} \cdots \xrightarrow{i_{n-1}} \overline{k\S}_{\le n} \xrightarrow{\phantom{x}i_n\phantom{x}} \cdots \]
and the 2-rig maps
\[    \A^{\boxtimes\, 0} \xleftarrow{\phantom{x}\pi_0\phantom{x}} \cdots \xleftarrow{\phantom{x} \pi_{N-1}} 
\A^{\boxtimes N} \xleftarrow{\phantom{x}\pi_N} \cdots \]
where $\pi_N$ sends the generating subline objects $s_1, \dots, s_N$ to themselves and sends $s_{N+1}$ to zero. Using the former, it is easy to see that $\ksbar$ is the colimit of the linear categories $\overline{k\S}_{\le n}.$   Using the latter, we can define a 2-rig $\Ainfty$ that is some kind of limit of the 2-rigs $\A^{\boxtimes N}$. There are some subtleties here that we must address!  But in the end, we shall obtain a map of 2-rigs
\[    F \maps \ksbar \to \Ainfty  \]
and prove in \cref{thm:splitting2} that it is an extension. 

Let us address the subtleties. Clearly we need to decide whether to define $\Ainfty$ as a 1-categorical or 2-categorical limit of the 2-rigs $\A^{\boxtimes N}$---and if the latter, what kind. It turns out that a 1-categorical limit suffices. However, we need to take \emph{gradings} into account. Most of the 2-rigs and 2-rig maps that we have been discussing are in fact graded in an appropriate sense. By taking the limit in the category of graded 2-rigs, we avoid getting a limit that is `too large'.

To get a sense for why, note that this issue shows up already at the level of Grothendieck rings. Since $\A^{\boxtimes N}$ is the 2-rig of $\N^N$-graded vector spaces of finite total dimension, we have
\[    K(\A^{\boxtimes N}) \cong \Z[x_1, \dots , x_N] . \]
The ring homomorphism
\[  K(\pi_N) \maps \Z[x_1, \dots, x_{N+1}] \to \Z[x_1, \dots, x_N] \]
sends the generators $x_1, \dots , x_N$ to themselves and sends $x_{N+1}$ to zero. Homomorphisms of this sort give a diagram
\[  \Z \leftarrow \cdots \leftarrow  \Z[x_1, \ldots, x_N] \leftarrow \Z[x_1, \ldots x_N, x_{N+1}] \leftarrow \cdots \]
Practically by definition, the limit of this diagram in the category of rings, say $R$, consists of sequences of polynomials $P_N \in \Z[x_1, \dots, x_N]$ such that setting $x_{N+1}$ equal to zero in $P_{N+1}$ gives $P_N$. We can describe any such sequence as a possibly infinite formal sum of monomials in the variables $x_i$, for example
\[  P(x_1, x_2, \dots ) =   x_1 + 2x_1 x_2 + 3x_1 x_2 x_3 + \cdots .\]
We recover the polynomial $P_N$ from such a formal sum by setting all variables $x_i$ with $i > N$ equal to zero. In this description, the ring $R$ consists of all formal sums of monomials in the $x_i$ containing only finitely many monomials in any chosen finite set of variables. 

However, the rings $\Z[x_1, \dots, x_N]$ are all graded by degree, and the maps in the above diagram are homomorphisms of graded rings, so we can also take its limit in the category of graded rings. The result is some graded ring $S$. One can show that $S$ is a proper subring of $R$: it contains only those sequences $P_N$ that have \emph{bounded degree}. Thus, we may think of $S$ as the ring of formal sums of monomials of bounded degree in the variables $x_i$.
For example, $S$ does not contain the formal sum $P$ shown above, but it does contain all the symmetric functions, such as the $n$th power sum
\[  p_n(x_1, x_2, \dots) =  \sum_{i = 1}^\infty x^n_i  , \]
the $n$th elementary symmetric function
\[  e_n(x_1, x_2, \dots ) = \sum_{i_1 < \cdots < i_n} x_{i_1} \cdots x_{i_n}, \]
and the $n$th complete symmetric function
\[   h_n(x_1, x_2, \dots ) = \sum_{i_1 \le \cdots \le i_n} x_{i_1} \cdots x_{i_n} \]
for all $n \in \N$. In fact, the elements of $S$ fixed by all permutations of the variables $x_i$ form precisely the ring of symmetric functions \cite[Sec.\ I.2]{Macdonald}.

All this suggests that we should take the limit of the diagram
\[    \A^{\boxtimes\, 0} \xleftarrow{\phantom{x} \pi_0\phantom{x}} \cdots \xleftarrow{\phantom{x} \pi_{N-1}} 
\A^{\boxtimes N} \xleftarrow{\phantom{x}\pi_N} \cdots \]
not in the category of 2-rigs, but in the category of \emph{graded} 2-rigs. 

We present a careful treatment of graded 2-rigs in \cref{app:graded_2-rigs}. Here we merely describe the 2-category of $\N$-graded 2-rigs, and give the key examples we need.  We begin with some facts about Cauchy complete $k$-linear categories.

\begin{defn}
Define $\Cauch\Lin$ to be the 2-category with
\begin{itemize}
        \item Cauchy complete $k$-linear categories as objects, 
        \item $k$-linear functors as morphisms, 
        \item $k$-linear natural transformations as 2-morphisms.
\end{itemize}
\end{defn}

\begin{lem}
\label{lem:CauchLin_has_coproducts}
The 2-category $\Cauch\Lin$ has coproducts in the 2-categorical sense.
\end{lem}

\begin{proof}
 We need to show that given an indexed family $(\C_\alpha)_{\alpha \in A}$ with $\C_\alpha \in \Cauch\Lin$, there is a Cauchy complete $k$-linear category $\bigoplus_\alpha \C_\alpha$ equipped with $k$-linear functors 
 \[   i_\beta \maps \C_\beta \to \bigoplus_\alpha \C_\alpha \]
 such that for all $\D \in \Cauch\Lin$ there is an equivalence of categories
\[ p \maps \Cauch\Lin(\bigoplus_\alpha \C_\alpha, \D) \to \prod_\alpha \Cauch\Lin(\C_\alpha, \D) \]
whose $\alpha$th component, say $p_\alpha$, is obtained by precomposing with $i_\alpha$.

The objects of $\bigoplus_\alpha \C_\alpha$ are tuples \((c_\alpha)_{\alpha \in A}\) where \(c_\alpha \in \C_\alpha\) and all but finitely many of the \(c_\alpha\) are zero
objects.  A hom-set for $\bigoplus_\alpha \C_\alpha$ is the vector space defined by
\[    (\bigoplus_\alpha \C_\alpha) (c_\alpha)_{\alpha \in A}), (d_\alpha)_{\alpha \in A}) = \bigoplus_\alpha \C_\alpha(c_\alpha, d_\alpha) .\]
Composition and units are defined using those in the categories $\C_\alpha$.
It is routine to check that $(c_\alpha \oplus d_\alpha)_{\alpha \in A}$ is the biproduct of $(c_\alpha)_{\alpha \in A}$ and $(d_\alpha)_{\alpha \in A}$, and that idempotents split in the underlying ordinary category of $\bigoplus_\alpha \C_\alpha$   since idempotents split in each category $\C_\alpha$.  Thus $\bigoplus_\alpha \C_\alpha$ is Cauchy complete. 

There are $k$-linear functors $i_\beta \maps C_\beta \to \bigoplus_\alpha \C_\alpha$ defined by $i_\beta(c) = (c_\alpha)_{\alpha \in A}$ where 
\[  c_\alpha = \left\{  \begin{array}{ccc} c & \text{if } \alpha = \beta \\
0 & \text{if } \alpha \ne \beta. 
\end{array} \right. \]
To show that the resulting $k$-linear functor
\[ p \maps \Cauch\Lin(\bigoplus_\alpha \C_\alpha, \D) \to \prod_\alpha \Cauch\Lin(\C_\alpha, \D) \]
defined as above is an equivalence, we can check that it has an pseudo-inverse (meaning an inverse up to isomorphism)
\[  q \maps \prod_\alpha \Cauch\Lin(\C_\alpha, \D) \to \Cauch\Lin(\bigoplus_\alpha \C_\alpha, \D)  \]
which sends any tuple of $k$-linear functors $F = (F_\alpha \maps \C_\alpha \to \D)_{\alpha \in A}$ to the $k$-linear functor $F^\vee \maps \bigoplus_\alpha \C_\alpha \to D$ defined by
\[  F^\vee\left((c_\alpha)_{\alpha \in A}\right) = 
\bigoplus_\alpha F_\alpha(c_\alpha) . \qedhere \]
\end{proof}

We use these coproducts to define graded 2-rigs:

\begin{defn} 
An \define{graded 2-rig} is a 2-rig $\R$ equipped with Cauchy complete linear subcategories $\R_n$, called \define{grades}, for which the inclusions
$i_n \maps \R_n \to \R$ induce an equivalence of Cauchy complete linear categories
\[    \bigoplus_{n \in \N} \R_n \xlongrightarrow{\sim} \R, \]
and such that the tensor product and unit of $\R$ respect this decomposition as follows:
\[    \otimes \maps \R_m \boxtimes \R_n \to \R_{m + n} , \qquad I \in \R_0. \]
\end{defn}

We are mainly interested in two examples:

\begin{example}
\label{ex:ksbar_grading}
We give the 2-rig $\ksbar$ the grading whose $n$th grade is $\overline{kS_n}$, the linear category of all finite-dimensional representations of $S_n$, which is naturally a $k$-linear subcategory of $\ksbar$. For details, see \cref{ex:ksbar_grading_app}. 
\end{example}

\begin{example}
\label{ex:An_grading}
We give the 2-rig $\A^{\boxtimes N}$ the grading whose $n$th grade, $\A^{\boxtimes N}_n$, is the full subcategory on all $n$-fold tensor products of the generating bosonic sublines $s_i$. For details, see \cref{ex:AN_grading_app}. In \cref{thm:free_2-rig_on_N_bosonic_sublines} we saw how to identify $\A^{\boxtimes N}$ with the 2-rig of $\N^N$-graded vector spaces of finite total dimension: the object 
\[         s_1^{\otimes m_1} \otimes \cdots \otimes s_N^{\otimes m_N} \in \A^{\boxtimes N} \]
corresponds to the $\N^N$-graded vector space
with the field $k$ in grade $(m_1, \dots, m_N)$ and $0$ in every other grade.  If we make this identification, we obtain not just an equivalence but an isomorphism of $k$-linear categories
\[     \A^{\boxtimes N}_n \cong [\N^N(n), \Fin\Vect] \]
where $\N^N(n)$ is the set of $N$-tuples of natural numbers whose sum is $n$.  This nuance becomes important in \cref{sec:Ainfty}.
\end{example}

We can similarly define morphisms and 2-morphisms between graded 2-rigs, obtaining a 2-category of graded rigs.

\begin{defn} 
A \define{map of graded 2-rigs} $F \maps \R \to \R'$ is a 2-rig map that sends each grade $\R_n$ into the corresponding grade $\R'_n$. A \define{2-morphism of graded 2-rigs} is a linear natural transformation $\alpha \maps F \To F'$ between maps of $\N$-graded 2-rigs $F, F' \maps \R \to \R'$.
\end{defn}

Again, we are mainly interested in two examples:

\begin{example}
The 2-rig map 
\[        \phi_{N-1} \maps \A^{\boxtimes N} \to \A^{\boxtimes N - 1} \]
sending the generating sublines $s_1, \dots , s_{N-1}$ to themselves and sending $s_N$ to zero is a map of graded 2-rigs. 
\end{example}

\begin{example}
The 2-rig map 
\[         F_N \maps \ksbar \to \A^{\boxtimes N} \]
sending the generating object $x \in \ksbar$ to the object $s_1 \oplus \cdots \oplus s_N$ is a map of graded 2-rigs.  This map is isomorphic to the composite
\[ \begin{tikzcd}
    \overline{k\S}
    \arrow[r,"A"]
    & 
    \Rep(\M(N,k)) 
    \arrow[r,"B"]
    &
    \Rep(k^N) \simeq \A^{\boxtimes N}.
\end{tikzcd}
\]
In \cref{lem:BAi_ess_inj} we saw that precomposing this map with the inclusion $i \maps \overline{k\mathsf{S}}_{\le n} \to \ksbar$ gives an extension when $N \ge n$.
\end{example}

\section{The 2-rig $\Ainfty$}
\label{sec:Ainfty}

Now we turn to the 2-rig that plays the starring role in our splitting principle. We take the liberty of calling it $\Ainfty$, for just as $S^\infty$ is sometimes used in topology to denote the colimit of the topological spaces $S^N$, this 2-rig is the limit of the graded 2-rigs $\A^{\boxtimes N}$.

To obtain an explicit description of $\Ainfty$, we shall define it as a \emph{strict} limit rather than a fully 2-categorical limit.  For this to make sense, we must choose models of the graded 2-rigs $\A^{\boxtimes N}$ where they are specified up to isomorphism rather than merely up to equivalence.  Following \cref{ex:An_grading}, we henceforth define $\A^{\boxtimes N}$ to be the graded 2-rig of $\N^N$-graded vector spaces of finite total dimension. 

\begin{defn}
$\Ainfty$ is the strict limit of the diagram of graded 2-rigs
\[    \A^{\boxtimes\, 0} \xleftarrow{\phantom{x} \phi_0\phantom{x}} \cdots \xleftarrow{\phantom{x} \phi_{N-2}}  A^{\boxtimes N-1} \xleftarrow{\phantom{x} \phi_{N-1}} 
\A^{\boxtimes N} \xleftarrow{\phantom{x}\phi_N} \cdots \]
where $\phi_{N-1}$ sends the generating sublines $s_1, \dots , s_{N-1} \in \A^{\boxtimes N}$ to the like-named objects in $\A^{\boxtimes N - 1}$ and sends $s_N$ to zero.
\end{defn}
 
To describe $\Ainfty$ explicitly, we introduce the following sets: 

\begin{itemize}
    \item $\N^{(\infty)}$ is the set of sequences $(m_1, m_2, \dots)$ of natural numbers (including $0$) whose sum $m_1 + m_2 + \cdots$ is finite. 
    \item $\N^{(\infty)}(m)$ is the subset of $\N^{(\infty)}$ with sequences $(m_1, m_2, \ldots)$ whose sum is $m$. 
    \item $\N^N(m)$ is the set of $N$-tuples of natural numbers whose sum is $m$.
\end{itemize}
There are inclusions
\[  
\begin{array}{rrccl}
    i_{N,m} &\maps& \N^N(m) &\to& \N^{N+1}(m) \\
    && (m_1, \ldots, m_N) &\mapsto & (m_1, \ldots, m_N, 0) 
\end{array}
\]
and $\N^{(\infty)}(m)$ is the colimit in $\Set$ of this sequence of inclusions:
\[\N^{(\infty)}(m) \cong \lim_{\longrightarrow} \N^N(m) .\]
Precomposing with $i_{N,m}$ gives a $k$-linear functor 
\[
\phi_{N, m} \maps [\N^{N+1}(m), \Fin\Vect] \to [\N^N(m), \Fin\Vect]
\] 
with
\[   \phi_{N, m}(G)(m_1, \ldots, m_N) = G(m_1, \ldots, m_N, 0) \] 
Recalling from \cref{ex:An_grading} the isomorphism
\[    \A^{\boxtimes N}_m \cong [\N^N(m), \Fin\Vect ], \]
we see that $\phi_{N,m}$ in fact is the restriction of the graded 2-rig map $\phi_N \maps \A^{\boxtimes (N+1)} \to \A^{\boxtimes N}$ to the $m$th grade. 

\begin{lem}
\label{lem:Ainfty_mth_grade}
As a Cauchy complete linear category, the $m$th grade of $\Ainfty$ is isomorphic to $[\N^{(\infty)}(m), \Fin\Vect]$.
\end{lem}

\begin{proof}
Taking the limit as $N \to \infty$ we have
\[
\begin{array}{cccl}
   \qquad \qquad \qquad \qquad & \Ainfty_m
     &\cong& 
    \displaystyle{ \lim_{\longleftarrow} A^{\boxtimes N}_m }
    \\
    &&\cong& 
    \displaystyle{ \lim_{\longleftarrow} \; [\N^N(m), \Fin\Vect] }
    \\
    &&\cong&  
    \displaystyle{ [\lim_{\longrightarrow} \N^N (m) , \Fin\Vect] }
    \\
    &&\cong&  
    \displaystyle{ [\N^{(\infty)}(m), \Fin\Vect] } . \qquad \qquad \qquad \qquad \qedhere
\end{array}
\]
\end{proof}

To flesh out the graded 2-rig structure in this description of $\Ainfty$, we need to describe its monoidal structure in terms of $k$-linear functors 
\[
\otimes \maps \Ainfty_p \boxtimes \Ainfty_q \to \Ainfty_{p+q}.
\]
Let $F \in \Ainfty_p$ and $G \in \Ainfty_q$, and put $m = p+q$. Given $(m_1, m_2, \ldots) \in \N^{(\infty)}(m)$, there are only finitely many $(p_1, p_2, \ldots) \in \N^{(\infty)}(p)$ and $(q_1, q_2, \ldots) \in \N^{(\infty)}(q)$ such that $m_i = p_i + q_i$ for all $i$. Accordingly, define 
\[(F \otimes G)(m_1, m_2, \ldots) = \bigoplus_{(m_i) = (p_i) + (q_i)} F((p_i)) \otimes G((q_i))\]
where the $(p_i)$ and $(q_i)$ appearing in the coproduct indexing belong to $\N^{(\infty)}(p)$ and $\N^{(\infty)}(q)$, respectively. This gives the graded 2-rig structure on $\Ainfty$. 

In $\Ainfty$, there is for each $n \in \N$ an object of grade 1 that we shall call $s_n$, arising from the like-named objects of grade 1 in all the 2-rigs $\A^{\boxtimes N}$ for $N \ge n$.  Alternatively, we can use \cref{lem:Ainfty_mth_grade} to describe $s_n$ as the functor $L_n \maps \N^{\infty}(1) \to \Fin\Vect$ that sends the object 
\[   (0, \dots, 0, \underbrace{1}_{n\textrm{th place}}, 0, \dots) \in \N^{\infty}(1) \]
to $k$ and sends all other objects to zero.  One can check by explicit calculation that each object $s_n \in \Ainfty$ is a bosonic subline.

\begin{lem}
The 2-rig $\Ainfty$ contains an object of grade 1 that is the coproduct of all the bosonic sublines $s_1, s_2, \ldots \in \Ainfty$.
\end{lem}

\begin{proof} 
By \cref{lem:Ainfty_mth_grade}, we have an isomorphism of Cauchy complete linear categories
\[       \Ainfty_1 \cong [\N^\infty(1) , \Fin\Vect]  .\]
In the latter, the coproduct of the functors $L_n$ exists: it is the constant functor $L \maps \N^{(\infty)}(1) \to \Fin\Vect$ with value $k$. Thus, the coproduct of all the objects $s_n$ exists in $\Ainfty_1$.   Since $\Ainfty$ is the coproduct in $\Cauch\Lin$ of the grades $\Ainfty_m$ for $m \in \N$, this coproduct of the objects $s_n$ is also their coproduct in $\Ainfty$. 
\end{proof}

We denote the coproduct of all the bosonic sublines $s_n \in \Ainfty$ by $s_1 \oplus s_2 \oplus \cdots$.  By the universal property of $\ksbar$, up to isomorphism there exists a unique graded 2-rig map 
\[
F \maps \ksbar \to \Ainfty
\]
that takes the generator $x$ to this object $s_1 \oplus s_2 \oplus \cdots$.  This 2-rig map fits into a triangle 
\[
\begin{tikzcd}
    & \Ainfty \ar[d, "\pi_N"] \\ \ksbar \ar[r, "F_N"'] \ar[ur, "F"] & \A^{\boxtimes N} 
\end{tikzcd}
\]
which commutes when applied to the generator $x$, and therefore commutes up to a unique natural isomorphism. We next turn to the key properties of $F$.

\section{The splitting principle}
\label{sec:splitting}

Our `splitting principle' says that we can extend the free 2-rig on one object to a 2-rig in which this object becomes an infinite coproduct of bosonic sublines. More precisely:

\begin{thm}
\label{thm:splitting2}
    The 2-rig map $F \maps \ksbar \to \Ainfty$ sending the generator $x \in \ksbar$ to the coproduct $s_1 \oplus s_2 \oplus \cdots \in \Ainfty$ is an extension.
\end{thm}

\begin{proof} 
First, we show that for each $n \in \N$, the composite
\[ 
\begin{tikzcd}
    \ksbar_{\le n} 
    \arrow[r,"i"]
    &
    \ksbar
    \arrow[r,"F"]
    &
    \Ainfty
\end{tikzcd}
\]
is an extension. By \cref{lem:extension} it  suffices to show that following this composite with any further map gives an extension, so let us use $\pi_N \maps \Ainfty \to \A^{\boxtimes N}$ for $N \ge n$. But this further composite
\[ 
\begin{tikzcd}
    \ksbar_{\le n} 
      \arrow[r,"i"]
    &
    \ksbar
    \arrow[r,"F"]
    &
 \Ainfty
  \arrow[r,"\pi_N"]
  &
 \A^{\boxtimes N}
\end{tikzcd}
\]
is naturally isomorphic to
\[ 
\begin{tikzcd}
    \ksbar_{\le n} 
      \arrow[r,"i"]
    &
    \ksbar
    \arrow[r,"F_N"]
    &
 \A^{\boxtimes N}
\end{tikzcd}
\]
which in turn is naturally isomorphic to the composite
\[
\begin{tikzcd}
    \ksbar_{\le n}
    \arrow[r,"i"]
    & 
    \ksbar
    \arrow[r,"A"]
    & 
    \Rep(\M(N,k)) 
    \ar[r,"B"]
    &
    \Rep(k^N) \simeq \A^{\boxtimes N}
    \end{tikzcd}
\]
because both $F_N$ and $B \circ A$ send the generating object $x \in \ksbar$ to $s_1 \oplus \cdots \oplus s_N$. Finally, this last composite is an extension because we saw in \cref{lem:B_ess_inj} that $B$ is an extension and we saw in \cref{thm:Ai_ess_inj} that $A \circ i$ is an extension when $N \ge n$.

Next we show that $F$ itself is an extension.
We have just seen that $F$ is faithful, conservative and essentially injective when restricted to each subcategory $\overline{k \mathsf{S}}_{\le n}$. We need to show that $F$ itself has these three properties. For this, the key is that $\ksbar$ is the colimit of the subcategories $\overline{k \mathsf{S}}_{\le n}$, so any finite set of objects and morphisms of $\ksbar$ lies in some $\overline{k \mathsf{S}}_{\le n}$. The three properties of $F$ are then straightforward:

(Fa) To show $F$ is faithful, assume $f, g \maps y \to z$ in $\ksbar$ have $F(f) = F(g)$. We know $f$ and $g$ are in $\overline{k \mathsf{S}}_{\le n}$ for some $n$. Since $F$ is faithful restricted to this subcategory we conclude $f = g$. Thus $F$ is faithful.

(Co) To show $F$ is conservative, assume $f \maps y \to z$ in $\ksbar$ is such that $F(f)$ is an isomorphism. We know $f$ is in  $\overline{k \mathsf{S}}_{\le n}$ for some $n$. Since $F$ is conservative when restricted to this subcategory, we conclude $f$ is an isomorphism. Thus $F$ is conservative.

(Es) To show $F$ is essentially injective, assume $x,y \in \ksbar$ have $F(x) \cong F(y)$. We know $x, y \in \overline{k \mathsf{S}}_{\le n}$ for some $n$. Since $F$ is essentially injective when restricted to this subcategory, we conclude $x \cong y$. Thus $F$ is essentially injective.
\end{proof}

We hope that \cref{thm:splitting2} is a special case of a more general splitting principle.

\begin{conj}
\label{conj:main_conjecture}
    Let $\R$ be a 2-rig and $r \in \R$ an object of finite bosonic subdimension. Then there exists a 2-rig $\R'$ and a map of 2-rigs $E \maps \R \to \R'$ such that:
    \begin{enumerate}
    \item $E(r)$ splits as a direct sum of finitely many bosonic subline objects.
    \item $E \maps \R \to \R'$ is a \define{2-rig extension}, i.e., it is faithful, conservative, and essentially injective.
    \item $K(E) \maps K(\R) \to K(\R')$ is  injective.  
    \end{enumerate}
\end{conj}
\noindent
Item (3) is one of the main classical applications of the splitting principle: to prove an equation involving $\lambda$-ring operations applied to some element $r \in K(\R)$ of finite bosonic subdimension, it suffices to prove the corresponding equation for $E(r) \in K(\R')$, where $E(r)$ splits as a sum of bosonic subline objects.

Here is a possible strategy for proving \cref{conj:main_conjecture}.  Because $\ksbar$ is the free 2-rig on one object $x$, there is a 2-rig map $\hat r \maps \ksbar \to \R$ with $\hat r(x) = r$, and this map is unique up to isomorphism.   Consider the following `2-pushout'---or more precisely, iso-cocomma object:
     \[
    \begin{tikzcd}
        \ksbar 
        \arrow[r, "F"]
        \arrow[d, "\hat r"']
        \arrow[dr, phantom, "\ulcorner", pos = 0.9]
        &
        \Ainfty
        \arrow[d, "G"]
        \\
        \R \arrow[r, "E"'] &
        \R'
    \end{tikzcd}
    \] 
Then the conjecture will follow if we can show:
    \begin{enumerate}
        \item $E(r)$ is a finite coproduct of bosonic subline objects.
        \item $E \maps \R \to \R'$ is a 2-rig extension.
        \item If for some $r_1,r_2 \in \R$, $r' \in \R'$ we have $E(r_1) \oplus r' \cong E(r_2) \oplus r'$ then for some $n$ we have $E(r_1) \oplus I^n \cong E(r_2) \oplus I^n$.
    \end{enumerate}
Thanks to \cref{lem:injectivity_on_K-theory}, item (3) implies that $E \maps \R \to \R'$ induces an injection of $\lambda$-rings, $K(E) \maps K(\R) \hookrightarrow K(\R')$.  

\section{Symmetric functions}
\label{sec:symmetric_functions}

In \cite{Schur} we showed that the Grothendieck group $K(\R)$ of any 2-rig $\R$ becomes a $\lambda$-ring in a functorial way, and that just as $\ksbar$ is the free 2-rig on one generator, $K(\ksbar)$ is the free $\lambda$-ring on one generator.  The 2-rig map 
\[    F \maps \ksbar \to \Ainfty \]
gives a $\lambda$-ring homomorphism
\[    K(F) \maps K(\ksbar) \to K(\Ainfty).\]

Now we conclude this paper by proving that:
\begin{itemize}
\item $K(F)$ is injective, so its image is isomorphic to $K(\ksbar)$.
\item $K(\Ainfty)$ is the ring of formal power series of bounded degree on countably many variables $x_1, x_2, x_3, \dots$.
\item The image of $K(F)$ is the ring of `symmetric functions', meaning those formal power series of bounded degree on countably many variables $x_i$ that are invariant under all permutations of these variables.
\end{itemize}
It follows that the ring of symmetric functions is the free $\lambda$-ring on one generator. This is, of course, well-known \cite{Hazewinkel}, but here we see it as a decategorified spinoff of results on 2-rigs.

\begin{thm}
\label{thm:K(Ainfty)}
As a graded ring, $K(\Ainfty)$ is isomorphic to the ring of formal power series of bounded total degree in countably many variables, graded by total degree.
\end{thm}

\begin{proof}
We calculate $K(\Ainfty)$ graded by grade.  We have isomorphisms of abelian groups
\begin{align*}
    K(\Ainfty_m) 
    &\cong K([\N^{(\infty)}(m), \Fin\Vect]) 
    \\ & \cong K\left(\textstyle{\prod_{(m_i) \in \N^{(\infty)}(m)}} \Fin\Vect\right) 
    \\& \cong \textstyle{\prod_{(m_i) \in \N^{(\infty)}(m)}} K(\Fin\Vect) 
    \\& \cong \textstyle{\prod_{(m_i) \in \N^{(\infty)}(m)}}\mathbb{Z} 
    \\& \cong \mathbb{Z}[[x_1, x_2, \ldots]]_m
\end{align*}
where $x_i$ represents the isomorphism class $[s_i]$ of the functor $s_i \maps \N^{(\infty)}(1) \to \Fin\Vect$ that takes the sequence $(0, \ldots, 0, 1, 0, \ldots)$ (with $1$ in the $i^{th}$ place) to $k$, and all other sequences to $0$. By \cref{lem:K_preserves_coproducts}, $K$ preserves coproducts.  Thus we have
\[ K(\Ainfty) = K(\bigoplus_m \Ainfty_m) \cong \bigoplus_m K(\Ainfty_m) \cong \bigoplus_m \mathbb{Z}[[x_1, x_2, \ldots]]_m.
\]
At right we have the graded ring of formal power series of bounded degree in countably many variables. $K(\Ainfty)$ is
isomorphic to this not only as graded abelian group, but as a graded ring, thanks to our description of the graded 2-rig structure on $\Ainfty$ near the end of \cref{sec:Ainfty}.
\end{proof}

In \cite[Sec.\ 6]{Schur} we explained how taking the Grothendieck group defines a 2-functor
\[    K \maps \Cauch\Lin_0 \to \Ab \]
where \(\Cauch\Lin_0\) is the 2-category with 
\begin{itemize}
\item Cauchy complete $k$-linear categories as objects, 
\item $k$-linear functors as morphisms, 
\item $k$-linear natural \emph{isomorphisms} as 2-morphisms,
\end{itemize}
and we treat the category $\Ab$ of abelian groups as a 2-category with only identity morphisms.  In particular we described $K$ as the composite
\[
\begin{tikzcd}
    \Cauch\Lin_0
    \arrow[r, swap, "J"] \arrow[rr, bend left, "K"]
    &
    \CMon 
    \arrow[r, swap, "\Z \otimes_\N -"]
    &
    \Ab
\end{tikzcd}
\]
where $J$ sends any Cauchy complete $k$-linear category to its set of isomorphism classes of objects, which is a commutative monoid with binary coproduct as addition, and $\Z \otimes_\N -$ is group completion.

\begin{lem}
\label{lem:K_preserves_coproducts} 
$K \maps \Cauch\Lin_0 \to \Ab$ preserves coproducts.
\end{lem}

\begin{proof}
It is clear that $\mathbb{Z} \otimes_\mathbb{N} -$ preserves coproducts. Now we argue that $J$ takes coproducts in $\Cauch\Lin_0$ to coproducts in $\CMon$. For this we must show that given a tuple of objects $(\C_\alpha)_{\alpha \in A}$ in $\Cauch\Lin$, the evident comparison map 
\[   \bigoplus_\alpha J(\C_\alpha) \to J(\bigoplus_\alpha \C_\alpha) \] 
is an isomorphism. 

To see this, note from \cref{lem:CauchLin_has_coproducts} that an isomorphism in $\bigoplus_\alpha \C_\alpha$ is a tuple $(f_\alpha \maps c_\alpha \to d_\alpha)_{\alpha \in A}$ where each $f_\alpha$ is an isomorphism in $\C_\alpha$ and all but finitely many of the objects $c_\alpha, d_\alpha \in C_\alpha$ are nonzero.  Thus, an element of $J(\bigoplus_\alpha \C_\alpha)$ amounts to the same thing as a tuple $([c_\alpha])_{\alpha \in A}$ where $[c_\alpha]$ is an isomorphism class in $\C_\alpha$ and all but finitely many of the objects $c_\alpha$ are nonzero.  But this is an element of $\bigoplus_\alpha J(\C_\alpha)$, so the comparison map is an isomorphism.
\end{proof}

\begin{defn} 
The ring of \define{symmetric functions}, $\Lambda$, is defined to be the subring of 
\[   \bigoplus_m \mathbb{Z}[[x_1, x_2, \ldots]]_m 
\]  
consisting of elements that are invariant under all permutations of the variables $x_i$.
\end{defn}

Thus, we have an inclusion of graded rings $\Lambda \subseteq K(\Ainfty)$, and using this we can show $\Lambda$ is isomorphic to $K(\ksbar)$:

\begin{thm}
\label{thm:injection_on_K-theory}
The 2-rig map $F \maps \ksbar \to \Ainfty$ induces an injective homomorphism of $\lambda$-rings
\[   K(F) \maps K(\ksbar) \to K(\Ainfty)  \]  
whose image is $\Lambda$.  
\end{thm}

\begin{proof}
The injectivity of $K(F)$ follows from Lemmas \ref{lem:injectivity_on_K-theory} and \ref{lem:Ainfty_cancellable}.  Here we show that the image of $K(F)$ is $\Lambda$.

Let $S_\infty$ be the group of all permutations of $\{1, 2, 3, \dots\}$.  This group acts on the set $\N^{(\infty)}(m)$ of sequences of natural numbers whose sum is $m$. This in turn gives a strict action of $S_\infty$ as automorphisms of the Cauchy complete linear category $[\N^{(\infty)}(m),\Fin\Vect]$, which by  \cref{lem:Ainfty_mth_grade} is isomorphic to the $m$th grade of $\Ainfty$.  By the explicit description of the graded 2-rig structure on $\Ainfty$ appearing directly after that lemma, it follows that these actions on each grade fit together to give a strict action of $S_\infty$ on $\Ainfty$ as graded 2-rig automorphisms.  Applying the functor $K$, this action becomes the action of $S_\infty$ on 
\[    K(\Ainfty) \; \cong \; \bigoplus_m \Z[[x_1, x_2, \dots]]_m \]
given by permuting variables.   The elements of $K(\Ainfty)$ fixed by this action of $S_\infty$ are precisely the symmetric functions.

The object $F(x) = s_1 \oplus s_2 \cdots \in \Ainfty$ is fixed up to isomorphism by $S_\infty$.  Since $\ksbar$ is generated as a 2-rig by $x$, it follows that every object $a$ in the essential image of $F$ is fixed up to isomorphism by $S_\infty$.   It follows that $[a] \in K(\Ainfty)$ is fixed by $S_\infty$, and is thus a symmetric function.   Thus, the image of $K(F)$ is contained in $\Lambda$.

For the reverse inclusion we can use the Fundamental Theorem of Symmetric Functions \cite[Sec.\ I.2]{Macdonald}, which implies that $\Lambda$ is generated as a ring by the elementary symmetric functions
\[  e_n(x_1, x_2, \dots ) = \sum_{i_1 < \cdots < i_n} x_{i_1} \cdots x_{i_n}. \]
Each elementary symmetric function is in the image of $K(F)$, because 
\[    e_n = [\Lambda^n(s_1 \oplus s_2 \oplus \cdots)]. \]
Thus, $\Lambda$ is contained in the image of $K(F)$.
\end{proof}

We state the following lemma in more generality than needed for \cref{thm:injection_on_K-theory}, with a view to making item (3) in \cref{conj:main_conjecture} as weak as possible while still giving an injection of Grothendieck groups.

\begin{lem}
\label{lem:injectivity_on_K-theory}
    Suppose $M \maps \R \to \S$ is a 2-rig map with the following properties:
    \begin{enumerate}
        \item $M$ is essentially injective
        \item If for some $r,r' \in \R, s \in \S$ we have $M(r) \oplus s \cong M(r') \oplus s$
    then for some $n$ we have $M(r) \oplus I^n \cong M(r') \oplus I^n$.
    \end{enumerate}
    Then $K(M) \maps K(\R) \to K(\S)$ is injective.
\end{lem}

\begin{proof} 
      Two elements $a,b$ of a commutative monoid $A$ give equal elements $\overline a = \overline b$ in the group completion $\overline A$ if and only if there exists an element $c \in A$ such that $a + c = b + c$.
      Given an element $r$ of a 2-rig $\R$, write $[r]$ for its isomorphism class and $\overline{[r]}$ for the corresponding element in $K(\R)$.
      Suppose $r, r' \in \R$ have $\overline{[M(r)]} = \overline{[M(r')]}$. Then there exists an object $s \in \S$ such that $[M(r)] + [s] = [M(r')] + [s]$, or equivalently $[M(r) \oplus s] = [M(r') \oplus s]$, or equivalently $M(r) \oplus s \cong M(r') \oplus s$. Condition (2) implies  $M(r) \oplus I^n \cong M(r') \oplus I^n$ for some $n$.  Since $M$ is a 2-rig map, we have $M(r \oplus I^n) \cong M(r' \oplus I^n)$. Since $M$ is essentially injective, we have $r \oplus I^n \cong r' \oplus I^n$, which gives $[r] = [r']$ in $J(\R)$, which in turn gives $\overline{[r]} = \overline{[r']}$, as desired.
\end{proof}

Condition (2) is always true when $S$ is the 2-rig of finitely generated projective modules of a commutative ring since then every $s \in \S$ is a summand of a finitely generated free module $I^n$.  Condition (2) is implied by the stronger condition one might call `cancellability in the image of $M$':
\begin{itemize} 
    \item If for some $r,r' \in \R, s \in \S$ we have $M(r) \oplus s \cong M(r') \oplus s$ then $M(r) \cong M(r')$.
\end{itemize}
This in turn is implied by cancellability in $\S$:
\begin{itemize} 
    \item If for some $t,t',s \in \S$ we have $t \oplus s \cong t' \oplus s$, then $t \cong t'$.
\end{itemize}
This last condition holds in the case of present interest, $\S = \Ainfty$.

\begin{lem}
\label{lem:Ainfty_cancellable}
If for some $t,t',s \in \Ainfty$ we have $t \oplus s \cong t' \oplus s$, then $t \cong t'$.
\end{lem}

\begin{proof}
Since $\Ainfty$ is the coproduct in $\Cauch\Lin$ of its grades $\Ainfty_m$, it suffices to prove this grade by grade, where by \cref{lem:Ainfty_mth_grade} we can use the isomorphism of Cauchy complete linear categories
\[     \Ainfty_m \cong [\N^{(\infty)}(m), \Fin\Vect]   .\]
Since coproducts are computed pointwise in $[\N^{(\infty)}(m), \Fin\Vect]$, the desired result then follows from the corresponding result in $\Fin\Vect$.
\end{proof}

\appendix

\section{The fermionic story}
\label{app:fermionic}

Our treatment has largely neglected the role of supersymmetry in the theory of 2-rigs, which should ultimately be taken into account. Given any Young diagram we can reflect it across the diagonal and get a new Young diagram, with rows of the original diagram becoming columns of the new one and vice versa. As we shall see, this reflection symmetry arises from a kind of involution on $\ksbar$. This involution switches the two 1-dimensional irreducible representations of $S_n$:
\[ \raisebox{1.2 em}{\yng(3)} \qquad \raisebox{1.5em}{$\mapsto$} \qquad \yng(1,1,1) \]
whose corresponding Schur functors are $S^n$ and $\Lambda^n$, respectively.
Since $S^n$ and $\Lambda^n$ are used to define the bosonic and fermionic versions of line objects, subline objects, dimension and subdimension, this suggests that these pairs of concepts should be treated on an equal footing, but we have not yet done so. As a small step in this direction, here we describe the free 2-rig on a fermionic subline object, the free 2-rig on a fermionic subline object, and the involution on $\ksbar$.

In \cref{thm:free_2-rig_on_bosonic_subline} we saw that the free 2-rig on a bosonic subline object is the monoidal $k$-linear category of $\N$-graded vector spaces of finite total dimension, equipped with the symmetry where $S_{s,s} \maps s \otimes s \to s \otimes s$ is the identity for any 1-dimensional vector space $s$ of grade 1. But the same monoidal $k$-linear category also admits another symmetry, determined by the fact that $S_{s, s} = -1_{s \otimes s}$. This gives a 2-rig we call $\widehat{\A}$. 

\begin{thm} 
\label{thm:free_2-rig_on_fermionic_subline}
$\widehat{\A}$ is the free 2-rig on a fermionic subline object. That is, given a 2-rig $\R$ containing a fermionic subline object $x$, there is a map of 2-rigs $F \maps \widehat{\A} \to \R$ with $F(s) = x$, and $F$ is unique up to isomorphism.
\end{thm} 

\begin{proof}
As a monoidal $k$-linear category $\widehat{\A}$ is the same as $\A$, and we again take $s$ to be any 1-dimensional vector space in grade $1$. However the symmetry on $\widehat{\A}$ introduces a sign change when permuting homogeneous elements of odd degree. Thus, the functor $F \maps \widehat{\A} \to \R$ defined by
\[        F(V) = \bigoplus_{n \ge 0} V_n \cdot x^{\otimes n} \]
is monoidal and $k$-linear as in  \cref{thm:free_2-rig_on_bosonic_subline}, but it is symmetric monoidal because the extra sign in the symmetry of $\widehat{\A}$ matches the sign change that occurs for the symmetry on $x^{\otimes j} \otimes x^{\otimes k}$ when both $j$ and $k$ are odd. One can also check that $F$ is unique up a monoidal natural isomorphism.
\end{proof}

\begin{example}
\label{ex:fermionic_subline_counterexample}
Not every fermionic subline object is a subobject of a line object. For since the tensor product in $\widehat{\A}$ is the usual tensor product of $\N$-graded vector spaces, the only line object in $\widehat{\A}$ is the tensor unit $I$. Thus, the fermionic subline object $s \in \widehat{\A}$ is not a subobject of any line object.
\end{example}

There is a similar story for line objects. In \cref{thm:free_2-rig_on_bosonic_line} we saw that the free 2-rig on a bosonic line object is the monoidal $k$-linear category of $\Z$-graded vector spaces of finite total dimension, with the symmetry where $S_{\ell,\ell}$ is the identity for any 1-dimensional vector space $\ell$ of grade 1. This monoidal $k$-linear category admits another symmetry determined by the fact that $S_{\ell, \ell} = -1_{\ell \otimes \ell}$. This gives a 2-rig we call $\widehat{\T}$. 

\begin{thm} 
\label{thm:free_2-rig_on_fermionic_line}
$\widehat{\T}$ is the free 2-rig on a fermionic line object. 
\end{thm} 

\begin{proof}
The proof follows the same pattern as that of \cref{thm:free_2-rig_on_fermionic_subline}.
\end{proof}

We conclude by explaining the `supersymmetry' involution on $\ksbar$.
In \cite[Sec.\ 7]{Schur} we defined a 2-rig $\G$ of $\mathbb{Z}_2$-graded Schur objects. 
The underlying category of $\G$ is the product $\ksbar \times \ksbar$. We write objects of $\G$ as $C = (C_0, C_1)$, and we call $C_0$ and $C_1$ the \define{bosonic} and \define{fermionic} parts of $C$. The tensor product on $\G$ is graded tensor product
\[
    (C_0, C_1) \otimes (D_0, D_1) = ((C_0 \otimes D_0) \oplus (C_1 \otimes D_1), \; (C_0 \otimes D_1) \oplus (C_1 \otimes D_0)).
\]
and the symmetry inserts a minus sign when switching two fermionic parts:
\[
    (S_{C, D})_0 = S_{C_0, D_0} \oplus -S_{C_1, D_1}, \quad 
    (S_{C, D})_1 = S_{C_0, D_1} \oplus S_{C_1, D_0} .
\]
There is an essentially unique 2-rig map 
\[  \phi_- \maps \ksbar \to \G \]
that sends the generating object $x \in \ksbar$ to the graded object $(0, x) \in \G$. There is also a functor
\[    T \maps \G \to \ksbar \]
sending any object $(C_0, C_1) \in \G$ to the direct sum $C_0 \oplus C_1 \in \ksbar$, and similarly for morphisms. This functor $T$ is monoidal and $k$-linear, but not a 2-rig map because it is not \emph{symmetric} monoidal. The composite
\[     \Omega = T \circ \phi_- \maps \ksbar \to \ksbar \]
is also a monoidal $k$-linear functor but not a 2-rig map.

To see that $\Omega$ is an involution, and better understand its properties, note from \cite[Thms.\ 9, 10]{Schur} that we can identify $\ksbar$ with the category of polynomial species, i.e.\ functors 
\[
F \maps k\S \to \Fin\Vect 
\]
for which all but finitely many values $F(n)$ are zero. If we describe $\Omega$ in these terms, a calculation shows that we have a natural isomorphism
\[    (\Omega F)(n) \cong \det(n) \otimes F(n) \]
where $\det(n)$ is the sign representation of $S_n$. This shows that $\Omega$ sends the object corresponding to any Young diagram to the object corresponding to the reflected version of that Young diagram. Furthermore, since $\det(n) \otimes \det(n)$ is the trivial representation, we have
\[         \Omega^2 \simeq 1_{\ksbar} \]
as monoidal $k$-linear functors. 

These facts suggest that most of our main theorems should have fermionic analogues. 

\section{More on graded 2-rigs}
\label{app:graded_2-rigs}

We gave a quick treatment of $\N$-graded 2-rigs in \cref{sec:graded_2-rigs}.  Here we put the theory of graded 2-rigs on a firmer and more general footing.

We start with some ordinary algebra. Let $M$ be a commutative monoid with identity element $e$. An $M$-graded vector space is a collection of vector spaces $V_m$, one for each $m \in M$; the category of $M$-graded vector spaces is a 2-rig. An $M$-graded algebra is a monoid with respect to the tensor product $\otimes$ of this 2-rig; equivalently, an $M$-graded algebra $R$ is an $M$-graded vector space $(R_m)_{m \in M}$ together with linear maps 
\[
R_m \otimes R_n \to R_{mn}, \qquad k \to R_e
\]
for all $m,n \in M$, satisfying appropriate associativity and unit conditions. As shown in 
\cref{lem:graded_over_commutative_monoid}, the 2-rig of $M$-graded vector spaces of finite total dimension is equivalent to the 2-rig of finite-dimensional comodules of the bicommutative bialgebra $kM$, where the commutative multiplication is used to give this 2-rig its symmetric monoidal structure. We can also drop the finite-dimensionality conditions here.

These ideas can be categorified in a straightforward way. For example, we have the following definition, which we state roughly on a first pass. As before, let $M$ be a commutative monoid with unit $e$. We define an `$M$-graded 2-rig' $\R$ to be a collection of Cauchy complete linear categories $\R_m$, one for each $m \in M$, together with linear functors 
    \[
    \R_m \boxtimes \R_n \to \R_{mn}, \qquad \Fin\Vect \to \R_e
    \]
for all $m,n \in M$, together with appropriate associators and unitors satisfying the usual coherence laws in a symmetric monoidal category. 

A deeper approach is to categorify the notion of comodule over a bialgebra, and describe gradings in terms of 2-comodules over 2-bialgebras. Interestingly, just as $M$-graded vector spaces are the same as comodules of the bialgebra $kM$, $M$-graded Cauchy complete linear categories turn out to be the same as 2-comodules of the 2-bialgebra $\overline{kM}$.

To pursue this approach, we must recall from \cite[Sec.\ 3]{Schur} that the 2-category $\Cauch\Lin$ is symmetric monoidal with the tensor product $\boxtimes$ described in the proof of Lemma 15 of that paper.   A symmetric pseudomonoid in $(\Cauch\Lin, \boxtimes)$ is the same as a 2-rig. 

Sch\"appi proved that given symmetric pseudomonoids $\R$ and $\R'$ in a symmetric monoidal 2-category, there is a natural way to make their tensor product $\R \boxtimes \R'$ into a symmetric pseudomonoid that is the coproduct of $\R$ and $\R'$ \cite[Thm.\ 5.2]{Schappi}.  He showed this construction gives a symmetric monoidal 2-category of symmetric pseudomonoids, which is cocartesian in the 2-categorical sense. Applying this construction to $(\Cauch\Lin, \boxtimes)$, it follows that $(\TRig, \boxtimes)$ is cocartesian symmetric monoidal. Alternatively we can treat $\boxtimes$ as the product in $\TRig\op$, which is cartesian symmetric monoidal. 

\begin{defn}
A \define{2-bialgebra} $\B$ is a pseudomonoid in $(\TRig\op, \boxtimes)$. We say a 2-bialgebra $\B$ is \define{cocommutative} if this pseudomonoid is symmmetric.
\end{defn}

More concretely, a 2-bialgebra is a 2-rig $\B$ equipped with linear functors called a \define{comultiplication} $\delta \maps \B \to \B \boxtimes \B$ and \define{counit} $\varepsilon \maps \B \to \Fin\Vect$ obeying the laws of a bialgebra up to coherent linear natural isomorphisms.

Next we introduce 2-comodules of 2-bialgebras. In general, any pseudomonoid $\mathsf{M}$ in a symmetric monoidal 2-category $(\mathbf{K}, \boxtimes)$ induces a pseudomonad $- \boxtimes \mathsf{M}$ on the underlying 2-category $\mathbf{K}$. Similarly, a pseudomonoid $\mathsf M$ in $(\mathbf{K}\op, \boxtimes)$ induces a pseudomonad $- \boxtimes \mathsf M$ on $\mathbf{K}\op$, which we can also call a `pseudocomonad' on $\mathbf{K}$. In particular, any 2-bialgebra $\B$ induces a pseudocomonad $- \boxtimes \B$ on $\Cauch\Lin$.

\begin{defn}
A \define{2-comodule} of a 2-bialgebra $\B$ is a pseudocoalgebra of the pseudocomonad $- \boxtimes \B$ on $\Cauch\Lin$.
\end{defn}

Unpacking this a bit, a 2-comodule of a 2-bialgebra $\B$ is a Cauchy complete linear category $\C$ equipped with a linear functor called a \define{coaction}
\[    \eta \maps \C \to \C \boxtimes \B \]
obeying the usual axioms for a comodule up to  `coassociator' and `right counitor' natural isomorphisms that obey versions of the pentagon and unitor equations.

We will be especially interested in 2-comodules that are also 2-rigs, where $\eta$ and the other structures just mentioned are compatible with the 2-rig structure. We can define these `2-rig 2-comodules' as follows. Suppose $\B$ is a 2-bialgebra. Then $- \boxtimes \B$ defines a pseudocomonad, not only on $\Cauch\Lin$, but on $\TRig$. 

\begin{defn}
A \define{2-rig 2-comodule} of a 2-bialgebra $\B$ is a pseudocoalgebra of the pseudocomonad $- \boxtimes \B$ on $\TRig$.
\end{defn}

In other words, a 2-rig 2-comodule of $\B$ is a 2-rig $\R$ that is a 2-comodule of $\B$ for which the coaction
\[    \eta \maps \R \to \R \boxtimes \B \]
is a morphism in $\TRig$ and the coassociator and right counitor are 2-morphisms in $\TRig$.

\begin{example}
\label{ex:kM_2-bialgebra}
For a commutative monoid $M$, the 2-rig $\overline{kM}$ discussed in \cref{lem:graded_over_commutative_monoid} acquires a 2-bialgebra structure from the diagonal $\Delta \colon M \to M \times M$ and the map to the terminal commutative monoid, $! \maps M \to 1$.  The diagonal induces the comultiplication $\delta$ on $\overline{kM}$ given by the composite
\[
\begin{tikzcd}
    \overline{kM}
    \arrow[r, "\overline{k\Delta}"] \arrow[rr, bend left, "\delta"]
    &
     \overline{k(M \times M)}
    \arrow[r, "\sim"]
    &
    \overline{kM} \boxtimes \overline{kM}
\end{tikzcd}
\]
where the equivalence comes from the proof of \cref{lem:rep_of_product}.   The map to the terminal commutative monoid induces the counit $\varepsilon$ on $\overline{kM}$ given by the composite
\[
\begin{tikzcd}
    \overline{kM}
    \arrow[r, "\overline{k!}"] \arrow[rr, bend left, "\varepsilon"]
    &
     \overline{k1}
    \arrow[r, "\sim"]
    &
    \Fin\Vect
\end{tikzcd}
\]
The resulting 2-bialgebra $\overline{kM}$ is cocommutative.
\end{example}

We may thus make the following definitions for any commutative monoid $M$:

\begin{defn}
An \define{$M$-graded Cauchy complete linear category} is a 2-comodule of $\overline{kM}$.
\end{defn}

\begin{defn}
An \define{$M$-graded 2-rig} is a 2-rig 2-comodule of $\overline{kM}$.
\end{defn}

It is worth spelling out these definitions in a more down to earth way.

\begin{lem}
\label{lem:M-graded_ccc}
For a Cauchy complete $k$-linear category $\C$ to be $M$-graded is equivalent to it being equipped with Cauchy complete $k$-linear subcategories $\C_m$, one for each $m \in M$, such that the inclusions $i_m \maps \C_m \to \C$ induce an equivalence in $\Cauch\Lin$:
\[   \bigoplus_{m \in M} \C_m  \xlongrightarrow{\sim} \C .\]
\end{lem}

\begin{proof} 
For a Cauchy complete $k$-linear category $\C$ to be $M$-graded means that it is equipped with a coaction 
\[     \eta \maps \C \to \C \boxtimes \overline{kM}. \] 
There is an equivalence of Cauchy complete $k$-linear categories 
\[   \alpha \maps \C \boxtimes \overline{kM} \xlongrightarrow{\sim} \bigoplus_{m \in M} \C \]
given as the composite
\[
    \C \boxtimes \overline{kM}
    \xlongrightarrow{\sim} 
    \C \boxtimes (\bigoplus_{m \in M} \Fin\Vect) 
    \xlongrightarrow{\sim}
    \bigoplus_{m \in M} (\C \boxtimes \Fin\Vect)
    \xlongrightarrow{\sim}
    \bigoplus_{m \in M} \C.
\]

We denote the composite of $\eta$ and $\alpha$ by $\eta'$:
\[
\begin{tikzcd}
    \displaystyle{ \C }
    \arrow[r, swap, "\eta"]
    \arrow[rr, bend left, "\eta'"]
    & 
    \displaystyle{ \C \boxtimes \overline{kM} }
    \arrow[r, swap, "\alpha"]
    &
    \displaystyle{ \bigoplus_{m \in M}  \C. }
\end{tikzcd}
\]
This map $\eta'$ is an equivalent, more tractable version of the comultiplication $\eta$.

Composing $\eta'$ and the projection $\pi_m \maps \bigoplus_{m \in M} \C \to \C$, we obtain a linear functor $p_m \maps \C \to \C$ which takes any object or morphism of $\C$ to its homogeneous component in degree $m$:
\[
\begin{tikzcd}
    \displaystyle{ \C }
    \arrow[r, swap, "\eta'"] \arrow[rr, bend left, "p_m"]
    &
   \displaystyle{  \bigoplus_{m \in M} \C }
    \arrow[r, swap, "\pi_m"]
    &
    \C.
\end{tikzcd}
\]
We define $\C_m$ to be the full image of $p_m$, which is a Cauchy complete $k$-linear subcategory of $\C$. 

By definition of the functors $p_m$, the map $\eta'$ takes an object $c \in \C$ to the tuple $(p_m(c))_{m \in M}$.  Thus, the linear functor
\[  \begin{array}{cccl} \theta \maps & \C & \longrightarrow & \displaystyle{ \bigoplus_{m \in M} \C_m }  \\ \\
                      &  c & \mapsto & \displaystyle{ (p_m(c))_{m \in M} }  
\end{array}
\]
followed by the evident inclusion
\[   \iota \maps \bigoplus_{m \in M} \C_m \to 
\bigoplus_{m \in M} \C \]
is isomorphic to $\eta'$:
\[  \eta' \cong \iota \circ \theta .\]

We now show that $\theta$ is an equivalence of Cauchy complete linear categories.  We do this by introducing the linear functor 
\[  \begin{array}{cccl} \sigma \maps & \displaystyle{ \bigoplus_{m \in M} \C } 
                        & \longrightarrow & \C \\ \\
                        & \displaystyle{ (c_m)_{m \in M}   } & \mapsto & \displaystyle{ \bigoplus_{m \in M} c_m }  
\end{array}
\]
and showing that $\sigma \circ \iota$ is a pseudo-inverse of $\theta$: that is, an inverse up to linear natural isomorphism.

To show that $\sigma \circ \iota$ is a left pseudo-inverse of $\iota \circ \theta$, consider the following diagram:
\[  \begin{tikzcd}
&& \displaystyle{ \bigoplus_{m \in M} \C_m } \ar[drr, "\iota"] 
\\
\C \ar[rr, "\eta "] \ar[ddrr, swap, "1"] \ar[urr, "\theta"] \ar[rrrr, bend left = 20, "\eta'"]
&& \C \boxtimes \overline{kM} \ar[d, "1 \boxtimes \, \varepsilon"] \ar[rr, "\alpha", "\sim"']
&& \displaystyle{ \bigoplus_{m \in M} \C} \ar[lldd, "\sigma"]
\\ 
&& \C \boxtimes \Fin\Vect \ar[d, "\sim" labl] \\
&& \C
\end{tikzcd} \]
where $\C \boxtimes \Fin\Vect \xrightarrow{\sim} \C$ is the right unitor in $\Cauch\Lin$.   Because $\C$ is a comodule of $\overline{k M}$, the triangle at left involving the coaction $\eta$ of $\overline{k M}$ on $\C$ and the counit $\varepsilon$ of $\overline{k M}$ commutes up to an isomorphism called the `counitor'.  The triangle at right commutes up to isomorphism thanks to the description of the counit $\varepsilon$ in \cref{ex:kM_2-bialgebra}.   We have already seen that the two triangles involving $\eta'$ commute 
up to isomorphism.  The diagram thus shows that $\sigma$ is
a left pseudo-inverse of $\iota \circ \theta$.  In terms of our concrete formulas for these functors, this fact says precisely that for $c \in \C$ there is a natural isomorphism
\[         c \cong \bigoplus_{m \in M} p_m(c)  .\]

In a similar way we can use the coassociator for the coaction of $\overline{k M}$ on $\C$ to show $\sigma \circ \iota$ is a right pseudo-inverse of $\theta$.  To begin with, note using \cref{ex:kM_2-bialgebra} that the linear functor
\[  1 \boxtimes \delta \maps \C \boxtimes \overline{kM} \to \C \boxtimes \overline{kM} \boxtimes \overline{kM} \]
induced by the comultiplication $\delta$ of $\overline{kM}$ corresponds to the map 
\[  \Delta \maps \bigoplus_{m \in M} \C \to \bigoplus_{(m, n) \in M \times M} \C \] 
that sends the tuple $(c_m)$ to $(\delta_{mn} c_m)_{(m, n)}$.   Again exploiting the equivalence $\alpha$, the coassociator gives a natural isomorphism filling in this square:
\[  \begin{tikzcd} 
\C \ar[rr, "\eta' "] \ar[d, swap, "\eta' "] && \displaystyle{ \bigoplus_{m \in M} \C } \ar[d, "\Delta"] \\ 
\displaystyle{  \bigoplus_{n \in M} \C } \ar[rr, swap, "\bigoplus_{n \in M} \eta' "] &&
\displaystyle{ \bigoplus_{(m, n) \in M \times M} \C} . \end{tikzcd}
\]
This directly translates to an $M \times M$-indexed array of natural isomorphisms 
\[  \delta_{mn} p_m(c) \cong p_m p_n(c). \]
In other words, $p_m p_m(c) \cong p_m(c)$ and $p_m p_n(c) \cong 0$ if $m \neq n$. We then obtain, for $(c_m)_{m \in M}  \in  \bigoplus_m p_m(\C)$, a series of isomorphisms 
\[ \begin{array}{rcll} 
(\theta \circ \sigma \circ \iota)(c_m)_{m \in M} & \cong & \theta (\bigoplus_m c_m) & \text{(definition of $\sigma, \iota $)}\\ 
& \cong & (p_n(\bigoplus_m c_m))_{n \in M} & \text{(definition of $\theta$)} \\  
& \cong & \left(\bigoplus_m p_n(c_m) \right)_{n \in M} & \text{($p_n$ preserves finite coproducts)} \\ 
& \cong & (c_n)_{n \in M} & \text{($c_m \in p_m(\C)$, $p_np_m \cong 0$ if $n \ne m$,} \\
& & & \text{ $p_m p_m \cong p_m$) }
\end{array} \] 
so that $\sigma \circ \iota$ is a right pseudo-inverse of $\theta$.

Conversely, suppose $\C$ is equipped with Cauchy complete $k$-linear subcategories $\C_m$, one for each $m \in M$, such that the inclusions $i_m \maps \C_m \to \C$ induce an equivalence 
\[    \C \simeq \bigoplus_{m \in M} \C_m .\] 
We wish to define a coaction $\eta$ of $\overline{k M}$ on $\C$.  To do this, let $p_m \maps \C \to \C_m$ be the projections onto these summands, and define
$\eta' \maps \C \to \bigoplus_{m \in M} \C$ by 
\[   \eta'(c) = (p_m(c))_{m \in M}. \]
Define $\eta \maps \C \to \C \boxtimes \overline{kM}$ by
\[     \eta = \beta \circ \eta' \]
where $\beta$ is any pseudo-inverse of the equivalence 
\[     \alpha \maps \C \boxtimes \overline{kM} \xrightarrow{\sim} \bigoplus_{m \in M} C \]
defined earlier.   To make $\eta$ into a coaction of $\overline{k M}$ on $\C$, we use the natural isomorphism
\[    c \cong \bigoplus_{m \in M} p_m(c) \]
to provide a counitor, and use the natural isomorphisms
\[   
p_m p_n (c) \cong \left\{ \begin{array}{cl}  p_n(c) & \text{if } m = n \\
        0      & \text{if } m \ne n
\end{array} \right. 
\]
to prove a coassociator, reversing the constructions above.  Finally, one can check that the counitor and coassociator obey the required coherence laws, and that this construction is inverse to the one described earlier, up to equivalence.
\end{proof}

\begin{lem}
\label{lem:M_graded_2-rig}
For a 2-rig $\R$ to be $M$-graded is equivalent to it being equipped with Cauchy complete $k$-linear subcategories $\R_m$, one for each $m \in M$, such that the inclusions $i_m \maps \R_m \to \R$ induce an equivalence in $\Cauch\Lin$:
\[   \bigoplus_{m \in M} \R_m  \xlongrightarrow{\sim} \R \]
and the tensor product and unit of $\R$ respect this decomposition:
\[    \otimes \maps \R_m \boxtimes \R_n \to \R_{m+n} , \qquad I \in \R_0. \]
\end{lem}

\begin{proof}
Since an $M$-graded 2-rig $\R$ is an $M$-graded Cauchy complete $k$-linear category, we get a decomposition 
\[   \R \simeq \bigoplus_{m \in M} \R_m  \]
as in \cref{lem:M-graded_ccc}.  Using the fact that the coaction $\eta$ of $\overline{k M}$ on $\R$ is a 2-rig map, one can show that
\[    \otimes \maps \R_m \boxtimes \R_n \to \R_{m+n} , \qquad I \in \R_0. \]
Conversely, any 2-rig equipped with this extra structure can be made into a 2-comodule of $\overline{kM}$ as in \cref{lem:M-graded_ccc}.  Using the fact that the tensor product and unit of $\R$ respect the grading, we can make the coaction $\eta$ into a 2-rig map.
\end{proof}

\begin{example}
\label{ex:kM_graded}
Any commutative 2-bialgebra $B$ is a 2-rig 2-comodule of itself, with its comultiplication $\delta \maps B \to B \boxtimes B$ as the coaction.  We saw in \cref{ex:kM_2-bialgebra} that for any commutative monoid $M$, the 2-rig $\overline{k M}$ is a commutative 2-bialgebra.  Thus $\overline{k M}$ is an $M$-graded 2-rig.
\end{example}

\begin{example}
\label{ex:A-comodules_N-graded}
Recall from \cref{thm:free_2-rig_on_N_bosonic_sublines} that the free 2-rig on a bosonic subline is $\A \simeq \overline{k\N}$. 
It follows that 2-rig 2-coalgebras over $\A$ are the same as $\N$-graded 2-rigs.  By \cref{ex:kM_graded}, $\A$ itself is an $\N$-graded 2-rig. 
\end{example}

To see how various 2-rigs studied in this paper are $\N$-graded, we can exploit their universal properties. 

\begin{example}
\label{ex:ksbar_grading_app}
The canonical $\N$-grading on the free 2-rig on one generator, $\ksbar$, comes from the coaction 
\[
\ksbar \to \ksbar \boxtimes \A
\]
that is the unique 2-rig map sending the generator $x$ to $x \otimes s$, where $s$ is the generating bosonic subline object of $\A$. The $n$th grade of $\ksbar$ is precisely the $k$-linear Cauchy completion of $x^n$, or in other words, the linear category $\overline{kS_n}$ of finite-dimensional representations of $S_n$. 

This grading induces a filtration of $\ksbar$ where the $n^{th}$ stage of the filtration is $\ksbar_{\leq n}$, consisting of finite coproducts (within $\ksbar$) of objects belonging to any grade $\ksbar_m$ with $m \leq n$. We have already seen this filtration play an important role in \cref{thm:Ai_ess_inj}, which is a key step toward the splitting principle.
\end{example}

\begin{example}
\label{ex:AN_grading_app}
Each 2-rig $\A^{\boxtimes N}$ has a unique $\N$-grading  
\[
\gamma_N \maps \A^{\boxtimes N} \to \A^{\boxtimes N} \boxtimes \A
\]
taking each generating bosonic subline $s_i$ of $\A^{\boxtimes N}$ to the subline $s_i \otimes s \in \A^{\boxtimes N} \boxtimes \A$. From this point of view, the 2-rig map 
\[   \phi_N \maps \A^{\boxtimes (N+1)} \to \A^{\boxtimes N}, \]
which sends $s_i$ to itself for $1 \le i \le N$ and $s_{N+1}$ to $0$, is actually a map of graded 2-rigs, since each path through the square 
\[
\begin{tikzcd}
    \A^{\boxtimes (N+1)} 
    \ar[r, "\gamma_{N+1}"] 
    \ar[d, "\phi_N"'] 
    & 
    \A^{\boxtimes (N+1)} \boxtimes \A 
    \ar[d, "\phi_N \boxtimes 1"] 
    \\
    \A^{\boxtimes N} 
    \ar[r, "\gamma_N"'] 
    & 
    \A^{\boxtimes N} \boxtimes \A
\end{tikzcd} 
\]
takes $s_i$ to $s_i \otimes s$ for $1 \leq i \leq N$ and takes $s_{N+1}$ to $0$. 

By a similar argument, the 2-rig map $\ksbar \to \A^{\boxtimes N}$ sending $x$ to $s_1 \oplus \cdots \oplus s_N$ also preserves the grading. This is essentially the map we called $B \circ A$ in \cref{sec:network}:
\[
\begin{tikzcd}
    \ksbar
    \ar[r,"A"]
    & 
    \Rep(\M(N,k)) 
    \ar[r,"B"]
    &
    \Rep(k^N) \simeq \A^{\boxtimes N}
\end{tikzcd}
\]
\end{example}

We conclude with a more speculative notion, suggested by the idea that $\ksbar$ is the `true' categorification of the polynomial algebra $k[x]$, since $\ksbar$ is the free 2-rig on one generator just as $k[x]$ is the free $k$-algebra on one generator. It might play a natural role in extensions of the current work. 

\begin{defn}
Let $\C$ be a Cauchy complete linear category. A \define{Young-grading} on $\C$ is a 2-comodule structure on $\C$ over the 2-bialgebra $\ksbar$, where the comultiplication is the co-operation identified as `comultiplication' in \cite[Thm.\ 4.4]{Schur}. 
\end{defn} 
The idea would be that for each Young diagram $D$, there is a corresponding homogeneous component of $\C$ in `degree' $D$.

\bibliographystyle{alpha}
\bibliography{references}

\end{document}